\documentclass[a4paper]{article}
\usepackage{amsmath}
\usepackage{amsfonts}
\usepackage{amssymb}
\usepackage{xcolor}
\usepackage{amsthm}
\usepackage{mathrsfs}
\usepackage{enumerate}
\usepackage{cite}

\usepackage{authblk}

\theoremstyle{plain}
\newtheorem{thm}{Theorem}

\theoremstyle{definition}
\newtheorem{defn}{Definition}[section]
\newtheorem{lem}{Lemma}[section]
\newtheorem{pro}[lem]{Proposition}
\newtheorem{corollary}[lem]{Corollary}

\newtheorem{remark}{Remark}[section]

\def\R{\mathbb{R}}
\def\Z{\mathbb{Z}}
\def\N{\mathbb{N}}
\def\K{\mathcal{K}}
\def\F{\mathcal{F}}
\def\V{\mathcal{V}}
\def\abs#1{|{#1}|}
\def\norm#1{\|{#1}\|}
\def\inp#1{\langle{#1}\rangle}
\def\supp{\mbox{supp}}
\def\half{{1 \over 2}}

\def\D{\mathcal{D}}
\def\epsilon{\varepsilon}

\def\wwP{{\Omega_{R,P,L}}}

\def\A#1{{A_{#1,P,L}}}
\def\hhG{\widehat{\widehat G\,}}
\def\oG{{\overline G}}

\def\ou{{\overline u}}
\def\hhu{\widehat{\widehat u}}

\def\dist{\mathop{\rm dist}\nolimits}
\def\wlimit{\rightharpoonup}

\def\be{\begin{equation}}
\def\ee{\end{equation}}

\title{Multi-bump solutions for logarithmic Schr\"odinger equations}

\author{
Kazunaga Tanaka\footnote{
The first author is partially supported by JSPS Grants-in-Aid for Scientific Research (B) (25287025)
and Waseda University Grant for Special Research Projects 2016B-120.}\\
Department of Mathematics\\ School of Science and Engineering, Waseda University\\
3-4-1 Ohkubo, Shijuku-ku, Tokyo 169-8555, Japan
\and
Chengxiang Zhang\footnote{
The second author is supported by China Scholarship Council and NSFC-11271201.
}\\
Chern Institute of Mathematics and LPMC, Nankai University\\
Tianjin 300071, China
}

\begin{document}

\date{}
\maketitle

\begin{abstract}
\noindent
We study spatially periodic logarithmic Schr\"odinger equations:
    \begin{equation}\tag{LS}
    \left\{
    \begin{aligned}
	&-\Delta u + V(x)u=Q(x)u\log u^2, \quad u>0\quad \text{in}\ \R^N, \\
    &\quad u\in H^1(\R^N),
    \end{aligned}\right. 
    \end{equation}
where $N\geq 1$ and $V(x)$, $Q(x)$ are spatially $1$-periodic functions of class $C^1$.
We take an approach using spatially $2L$-periodic problems ($L\gg 1$) and we show the
existence of infinitely many multi-bump solutions of $(LS)$ which are distinct
under $\Z^N$-action.
%In particular, we don't need any non-smooth variational theory.
\end{abstract}

\setcounter{equation}{0}
\section{Introduction}
We study the existence of solutions of the following spatially periodic logarithmic Schr\"odinger
equation:
    \begin{equation}\label{LS1}\tag{LS1}
    \left\{
    \begin{aligned}&-\Delta u + V(x)u=Q(x)u\log u^2\quad \text{in}\ \R^N,\\
        &u\in H^1(\R^N),
        \end{aligned}\right. 
    \end{equation}
where $N\in\N$ and $V$, $Q\in C^1(\R^N,\R)$ satisfy

\begin{enumerate}
\item[(A1)] $V(x)$, $Q(x)>0$ for all $x\in\R^N$;
\item[(A2)] $V(x)$, $Q(x)$ are $1$-periodic in each $x_i$ ($i=1,2,\cdots,N$), that is,
    \begin{eqnarray*}
        &&V(x_1,\cdots,x_i+1,\cdots,x_N)=V(x_1,\cdots,x_i,\cdots,x_N),\\
        &&Q(x_1,\cdots,x_i+1,\cdots,x_N)=Q(x_1,\cdots,x_i,\cdots,x_N)
    \end{eqnarray*}
for all $x=(x_1,x_2,\cdots,x_N)\in \R^N$ and $i=1,\cdots,N$.
\end{enumerate}

\begin{remark}
Substituting $u$ with $\lambda v$ ($\lambda>0$) in \eqref{LS1}, we get
	$$	-\Delta v+(V(x)-Q(x)\log \lambda^2)v=Q(x)v\log v^2.
	$$
Under the conditions (A2) and $Q(x)>0$ in $\R^N$, the positivity of $V(x)$ is not essential.
In fact, replacing $V(x)$ with $\widehat V(x)\equiv V(x)-Q(x)\log \lambda^2$,
we get $\widehat V(x)>0$ for a suitable $\lambda>0$ and thus we can recover positivity for any $V(x)$.
\end{remark}

\noindent
\eqref{LS1} has applications to physics (e.g. quantum mechanics, quantum optics etc.  See Zloshchastiev
\cite{Z} and references therein).

Formally solutions of \eqref{LS1} are characterized as critical points of 
    $$  I_\infty(u)=\half \int_{\R^N}|\nabla u|^2+V(x)u^2 -\half\int_{\R^N} Q(x)(u^2\log u^2-u^2).
    $$
However $\int_{\R^N} u^2\log u^2$ is not well-defined on $H^1(\R^N)$ due to the behavior of 
$s^2\log s^2$ as $s\sim 0$ and thus we cannot apply standard critical point theories to $I_\infty(u)$.
To overcome such difficulties, d'Avenia-Montefusco-Squassina \cite{dMS}, Squassina-Szulkin \cite{Sz}
and Ji-Szulkin \cite{JS}
applied non-smooth critical point theory for lower semi-continuous functionals.
In \cite{dMS}, d'Avenia, Montefusco and Squassina dealt with the case, where $V$ and $Q$ are constant,
and they showed the existence of infinitely many radially symmetric possibly sign-changing solutions.
They also showed the unique (up to translations) positive solution is so-called {\it Gausson} 
$\lambda e^{-\mu \abs x^2/2}$ for suitable constants $\lambda$, $\mu>0$.
See Bia{\l}ynicki-Birula and Mycielski \cite{BM1, BM2} for the {\it Gausson} and related topics.
See also \cite{dSZ} for logarithmic Sch\"odinger equaitons with fractional Laplacian.
In \cite{Sz}, Squassina and Szulkin considered
spatially periodic $V(x)$ and $Q(x)$ and they showed the existence of a ground state and infinitely many
possibly sign-changing solutions, which are geometrically distinct under $\Z^N$-action.  
Here we say that $u(x)$, $v(x)\in H^1(\R^N)$ are geometrically distinct if and only if 
$u(x+n)\not=v(x)$ for all $n\in\Z^N$.
In \cite{JS}, Ji and Szulkin also applied non-smooth variational framework to obtain a ground state and
infinitely many solutions for logarithmic Schr\"odinger equation under the setting: $Q(x)\equiv 1$ and
$\lim_{\abs x\to\infty}V(x)=\sup_{x\in\R^N}V(x)\in (-\infty,\infty)$ or 
$V(x)\to \infty$ as $\abs x\to \infty$.
To get infinitely many solutions, symmetry of functional $I_\infty(u)$ (i.e. evenness) and pseudo-index theories
are important in \cite{dMS,JS,Sz}. 
We also refer to Cazenave \cite{C} and Guerrero-L\'opez-Nieto \cite{GLN}
for approaches using a Banach space with a Luxemburg type norm or penalization.

%%%%

In this paper, we take another approach, which is inspired by Coti Zelati-Rabinowitz \cite{CZR3} and 
Chen \cite{Chen}, and we try to construct solutions of \eqref{LS1} through spatially $2L$-periodic solutions:
    \begin{equation}\label{eq:1.1}
    \left\{
    \begin{aligned}-&\Delta u + V(x)u=Q(x)u\log u^2 \quad \text{in}\ \R^N,\\
        &u(x+2Ln)=u(x) \qquad\qquad \ \ \text{for all}\ x\in\R^N\ \text{and}\ n\in\Z^N.
        \end{aligned}\right. 
    \end{equation}
That is, first we find a solution $u_L(x)$ of \eqref{eq:1.1} and second, after a suitable shift, we take
a limit as $L\to\infty$ to obtain a solution of \eqref{LS1}.
See also Rabinowitz \cite{R90} and Tanaka \cite{T} for earlier works.

The main purpose of this paper is to find multi-bump positive solutions of \eqref{LS1} with this approach.
In particular, we will show the existence of infinitely many geometrically distinct positive
solutions under $\Z^N$-action.  We note that in our argument symmetry of the functional is not
important.  Actually we find critical points of the following modified functional $J_\infty(u)$, which
is not even and whose critical points are non-negative solutions of \eqref{LS1} 
(see Section \ref{section:2} below).
    $$  J_\infty(u)=\half \int_{\R^N}|\nabla u|^2+V(x)u^2 -\int_{\R^N} Q(x)G(u).
    $$
Here $G(u)=\int_0^u g(s)\, ds$ and $g(u)$ is defined in \eqref{eq:2.1}. 

To state our main result, we need some preliminaries.  We set
    \begin{eqnarray*}
    &&\D=\{ u\in H^1(\R^N);\, \int_{\R^N}u^2\abs{\log u^2}<\infty\},\\
    &&{\cal N}_{\infty}=\{ u\in \D\setminus \{0\};\, 
        \int_{\R^N} \abs{\nabla u}^2+V(x)u^2-Q(x)g(u)u=0\}.
    \end{eqnarray*}
Any non-trivial non-negative solution of \eqref{LS1} lies in ${\cal N}_\infty$ and we consider
    \be\label{eq:1.2}
    b_\infty=\inf_{u\in {\cal N}_\infty} J_\infty(u).
    \ee
We denote
    \begin{eqnarray*}
        \K_\infty &=& \{ u\in \D;\, \text{$u$ is a positive solution of \eqref{LS1}}\},\\
        {[}J_\infty= c{]}_\infty &=& \{ u\in \D;\, J_\infty(u)=c\}, \\
        {[}J_\infty\leq c{]}_\infty &=& \{ u\in \D;\, J_\infty(u)\leq c\} \quad \text{for}\ c\in\R.
    \end{eqnarray*}
By the periodicity of $V$ and $Q$, we note that $\K_\infty$ is invariant under $\Z^N$-action.
Following \cite{Se, CZR1, CZR2}, to discuss multiplicity of positive solutions, we assume that 
for $\alpha>0$ small
    \begin{equation} \label{*}\tag{\textasteriskcentered}
    (\K_\infty\cap[J_\infty\leq b_\infty+\alpha]_\infty)/\Z^N\ \mbox{is finite}.
    \end{equation}
We note that if \eqref{*} does not hold, \eqref{LS1} has infinitely many geometrically distinct
solutions.  We choose a finite set 
    $$  \F_\infty=\{w^i\in \K_\infty\setminus\{ 0\};\, 1\leq i\leq m\}
    $$  
such that $w^i\neq w^j(\cdot+n)$ for $1\leq i,j\leq m$, $i\neq j$, $n\in \Z^N$, and 
$\K_\infty\cap[J_\infty=b_\infty]_\infty=\{w^i(\cdot+n);\, 1\leq i\leq m;\, n\in\Z^N\}$.

Now we can state our main result, which deals with $2$-bump solutions.

\begin{thm}\label{thm:1}
Assume (A1), (A2) and  \eqref{*}.  Then for any $r>0$, there exists $R_{r1}>0$ such that for any 
$P\in\Z^N$ with $\abs P>5R_{r1}$ and $\omega$, $\omega'\in\F_\infty$ ,
    $$  \K_\infty\cap B^{(\infty)}_{2r}(\omega+\omega'(\cdot-P))\neq\emptyset.
    $$
\end{thm}   

\noindent
Here we use notation:
    $$  B^{(\infty)}_r(u)=\{ v\in H^1(\R^N);\, \norm{v-u}_{H^1(\R^N)} <r\}
        \quad \text{for}\ u\in H^1(\R^N) \ \text{and} \ r>0.
    $$
We remark that such multi-bump type solutions were constructed via variational methods firstly 
by S\'er\'e \cite{Se} and Coti Zelati-Rabinowitz \cite{CZR2} for Hamiltonian systems and 
by Coti Zelati-Rabinowitz \cite{CZR1} for nonlinear elliptic equations.  
See also Alama-Li \cite{AL}, Liu-Wang \cite{LW1,LW2}, Montecchiari \cite{M} and S\'er\'e \cite{Se2}.

\begin{remark}
In Theorem \ref{thm:1}, we state the existence of 2-bump solutions.  
In a similar way, we can find $k$-bump solutions for any $k\in\N$.
\end{remark}

To show Theorem \ref{thm:1}, as we stated earlier, we will find a solution of \eqref{LS1}
through $2L$-periodic problems \eqref{PL} with large $L\in\N$:
    \begin{equation}\label{PL}\tag{PL}
        \left\{
        \begin{aligned}-&\Delta u + V(x)u=Q(x)g(u) \quad \text{in}\ \R^N,\\
        &u(x+2Ln) = u(x) \qquad\quad \ \text{for all}\ x\in\R^N \ \text{and}\ n\in\Z^N.
        \end{aligned}\right. 
    \end{equation}
That is, first we find a solution of \eqref{PL}, which has 2-bumps in $[-L,L)^N$ --- we call 
such a solution $2L$-periodic 2-bump ---  and second we take a limit as $L\to\infty$.
Such approaches were taken by Coti Zelati-Rabinowitz \cite{CZR3} and Chen \cite{Chen} for Hamiltonian systems
and strongly indefinite nonlinear Schr\"odinger equations.
They succeeded to construct multi-bump type periodic solutions through periodic multi-bump solutions.

Solutions of \eqref{PL} is characterized as critical points of 
    $$  J_L(u)= \half\int_{D_L} |\nabla u|^2+V(x)u^2 
        -\half\int_{D_L} Q(x)G(u):\, E_L\to\R,
    $$
where
    \begin{eqnarray*}
    D_L &=& [-L,L]^N,\\
    E_L &=& \{ u\in H^1_{loc}(\R^N);\, \text{$u(x)$ is $2L$-periodic in $x_i$ for all $i=1,2,\cdots,N$}\}.
    \end{eqnarray*}
In our problem, the approach using \eqref{PL} gives us a merit that the functional 
$\int_{D_L}u^2\log u^2$ and $J_L(u)$ are well-defined and of class $C^1$ since we work 
essentially in a bounded domain $D_L$ for \eqref{eq:1.1}.
Thus for \eqref{eq:1.1} we are in a better situation than the original problem in $\R^N$ and we can apply
the standard critical point theory to $J_L(u)$.
Our result for \eqref{PL} is the following theorem, in which we prove the existence of
$2L$-periodic 2-bump solutions.

\begin{thm}\label{thm:2}
Assume (A1), (A2) and \eqref{*}.  Then for any $r>0$, there exists $R_{r0}>0$ such that for any 
$L\in\N$, $P\in \Z^N$ with $\abs P>5R_{r0}$, $L>2\abs P$, and $\omega$, $\omega'\in\F_\infty$    
    $$  \K_L\cap B^{(L)}_r(\Phi_L(\psi_{R_{r0}}\omega)+\Phi_L({\psi_{R_{r0}}\omega'})(\cdot-P))\neq\emptyset.
    $$
\end{thm}       

\noindent
In Theorem \ref{thm:2}, we use notation:
    \begin{eqnarray*}
    \K_L &=& \{ u\in E_L;\, \text{$u$ is a positive solution of \eqref{PL}}\},\\
    B^{(L)}_r(u)&=&\{ v\in E_L;\, \norm{v-u}_{H^1(D_L)} <r\}
        \quad \text{for}\ u\in H^1(\R^N) \ \text{and} \ r>0.
    \end{eqnarray*}
and $\psi_{R_r}(x)$ is a suitable cut-off function around $0$ and $\Phi_L({\psi_{R_r}\omega})(x)$ is a 
$2L$-periodic extension of $\psi_{R_r}(x)$.  See Section \ref{section:2} below for a precise definition.

This paper is organized as follows:  In Section \ref{section:2}, first we introduce 
truncation of nonlinearity $u\log u^2$ and a modified functional together with its fundamental 
properties.  Second, for \eqref{PL} we observe that $J_L(u)$ has mountain pass geometry which is
uniform with respect to $L\in \N$.
In Section \ref{section:3}, we give a new concentration-compactness type result (Proposition \ref{pro:3.1}).
It enables us to take a limit as $L\to\infty$ in \eqref{PL} to obtain a solution of our original
problem \eqref{LS1}.  In Proposition \ref{pro:3.1}, uniform estimates of $\int_{D{L_j}} H(u_j)$ as well
as $\norm{u_j}_{H^1(D_{L_j})}$ for Palais-Smale type sequences $(u_j)_{j=1}^\infty$ ($u_j\in E_{L_j}$,
$L_j\to\infty$) is important.
In Section \ref{section:4}, we define mountain pass value $b_L$ for $J_L(u)$ and we study its
behavior as $L\to\infty$.  
It enables us to show that $b_\infty$ defined in \eqref{eq:1.2} is achieved and the existence of
the ground state (Theorem \ref{thm:3} and Corollary \ref{corollary:4.1}) is proved.  
Moreover, it is also important to show the existence of 2-bump solutions in Section \ref{section:5}.
In Section \ref{section:5}, we give proofs to Theorems \ref{thm:1} and \ref{thm:2}.
A gradient estimate in annular type neighborhood (Proposition \ref{pro:5.2} and deformation
lemma (Lemma \ref{lem:5.3}) are important.  Our deformation flow is constructed so that it keeps
    $$  \left\{ u\in E_L;\, \int_{D_L\setminus (B_R(0)\cup B_R(P))} \half(\abs{\nabla u}^2+V(x)u^2)
            -Q(x)G(u) \leq \rho\right\}
    $$
invariant for suitable constant $R\gg 1$ and $\rho>0$.  Together with the elliptic decay 
estimate in $D_L\setminus (B_R(0)\cup B_R(P))$, this property enables us to get a critical point
$u_L$ of $J_L(u)$ with an estimate on $\int_{D_L}u^2\log u^2$ independent of $L$.
In Appendix, we give a proof of our concentration-compactness type result (Proposition \ref{pro:3.1}).

%%%%%%%%%%%%%%%%%%%%%%%%%%%%

\setcounter{equation}{0}
\section{Preliminaries}\label{section:2}
\subsection{A modified problem}\label{section:2.1}

\noindent
To study \eqref{LS1}, we introduce the following modified problem:
    \begin{equation}\label{LS2}\tag{LS2}
        \left\{\begin{aligned}
        -&\Delta u+V(x)u =Q(x)g(u)\quad \text{in}\ \R^N,\\
        &u\in H^1(\R^N),
        \end{aligned}\right.
    \end{equation}
where $g(u)=-h(u)+f(u)$ and 
     $$ h(u)=\begin{cases}
        -u\log u^2, &\text{if}\ \abs u\leq e^{-1},\\
        2e^{-1}, & \text{if}\ u> e^{-1},\\
        -2e^{-1}, & \text{if}\ u<-e^{-1},
     \end{cases} \quad 
     f(u)=\begin{cases}
        0, &\text{if}\ u\leq e^{-1},\\
        2e^{-1}+u\log u^2, &\text{if}\ u> e^{-1}.
    \end{cases}
    $$
We note that
    \be\label{eq:2.1}
        g(u)=-h(u)+f(u)=\begin{cases}
        u\log u^2 &\text{if}\ u\geq -e^{-1},\\
        2e^{-1}   &\text{if}\ u< -e^{-1}.
        \end{cases}
    \ee
As we will see in Lemma \ref{lem:2.3}, non-zero solutions of \eqref{LS2} are positive solutions
of \eqref{LS1}.

\noindent
We set
    \begin{eqnarray*}
    &&H(u)=\int_0^u h(s)\,ds
    =\begin{cases}
        -\half u^2\log u^2+\half u^2    &\text{if}\ \abs u\leq e^{-1},\\
        \frac 2 e \abs u- \frac 1 {2e^2}     &\text{if}\ \abs u > e^{-1}, 
     \end{cases}\\
    &&F(u)=\int_0^u F(s)\,ds, \ G(u)=-H(u)+F(u)=\int_0^u g(s)\, ds.
    \end{eqnarray*}
We note that $H(u)$ is a convex part of $-\half u^2\log u^2+\half u^2$, which is sub-quadratic near $0$, 
and $F(u)$ is a super-quadratic part of $\half u^2\log u^2-\half u^2$, which is modified for $u<0$ so that
solutions of \eqref{LS2} correspond to non-negative solutions of \eqref{LS1}.

Here we give some properties of $g(u)$, $h(u)$ and $f(u)$.

\begin{lem}\label{lem:2.1}
\begin{enumerate}[(i)] 
\item $g\in C^1(\R\setminus\{ 0\})\cap C(\R)$ and 
$\frac{g(s)}{s}$ is strictly increasing in $(0,\infty)$ and strictly decreasing in $(-\infty,0)$;
\item $f\in C^1(\R)$ is positive and increasing in $[0,\infty)$. Moreover
\begin{enumerate}[(a)] 
\item $f(s)s\geq 2F(s)\geq 0$ for all $s\in \R$;
\item for any $p>2$ there is a constant $C_p>0$ depending on $p$, such that 
    $$  \abs{f(s)}\leq C_p\abs s^{p-1},\quad  F(s)\leq C_p\abs s^p \quad \text{for all}\ s\in\R;
    $$
\end{enumerate}
\item $h\in C(\R)$ is positive, increasing and concave in $[0,\infty)$ and 
$H(s)$ is convex on $\R$.  Moreover
\begin{enumerate}[(a)]
\item For all $s>0$ and $\theta \in [0,1]$
    \begin{eqnarray}
    &&\theta h(s) \leq h(\theta s),  \label{eq:2.2}\\
    &&\theta^2 H(s) \leq H(\theta s), \label{eq:2.3}\\
    &&\half h(s)s \leq H(s) \leq h(s)s; \label{eq:2.4}
    \end{eqnarray}
\item For all $s$, $t\in\R$ and $\theta \in (0,1]$
    \be\label{eq:2.5}
    \abs{h(s)t} \leq \theta H(s) +\frac 1 \theta H(t).
    \ee
\end{enumerate}
\end{enumerate}
\end{lem}

\noindent
We will give a proof of Lemma \ref{lem:2.1} at the end of this section.

\noindent
We note that for $u\in H^1(\R^N)$ 
    $$  \int_{\R^N} u^2 |\log u^2| < \infty \quad \text{if and only if}\quad
        \int_{\R^N} H(u) <\infty.
    $$
and we can write 
    $$  \D = \{ u\in H^1(\R^N);\, \int_{\R^N} H(u) <\infty\}.
    $$
By Lemma \ref{lem:2.1} we have

\begin{lem}\label{lem:2.2}
\begin{enumerate}[(i)] 
\item $u\mapsto \int_{\R^N} F(u);\, H^1(\R^N)\to \R$ is of class $C^1$;
\item For any $u$, $v\in \D$, $\int_{\R^N}h(u)v$ is well-defined;
\item ${C}_0^\infty(\R^N)$ is dense in $\D$ in the following sense:
for any $u\in \D$, there exists a sequence $\{ \varphi_n\}_{n=1}^\infty\subset {C}_0^\infty(\R^N)$
such that
    $$  \norm{\varphi_n-u}_{H^1} \to 0 \quad \text{and}\quad
        \int_{\R^N}H(\varphi_n) \to \int_{\R^N} H(u).
    $$
\end{enumerate}
\end{lem}

\noindent
We say that $u\in H^1(\R^N)$ is a {\it weak solution} of \eqref{LS1} (resp. \eqref{LS2}) if
$u$ satisfies $u\in \D$ and 
    \begin{eqnarray}    
    &&\int_{\R^N} \nabla u\nabla \varphi+V(x) u\varphi -Q(x)u\varphi\log u^2=0 \nonumber\\
      \Bigl(\text{resp.}  
        &&\int_{\R^N} \nabla u\nabla \varphi+V(x) u\varphi -Q(x)g(u)\varphi=0 \ \Bigr) \label{eq:2.6}\
    \end{eqnarray}
for all $\varphi\in {C}_0^\infty(\R^N)$. 

Next lemma ensures that any solution of \eqref{LS2} is a nonnegative solution of \eqref{LS1}.
In what follows, for $u\in\R$, we denote $u^+=\max\{u,0\}, u^-=\max\{-u,0\}$.

\begin{lem}\label{lem:2.3}
For $u\in \D$, the following statements are equivalent:
\begin{enumerate}[(i)] 
\item $u(x)$ is a non-negative weak solution of \eqref{LS1};
\item $u(x)$ is a weak solution of \eqref{LS2}.
\end{enumerate}
Moreover, any non-trivial solution of \eqref{LS2} are positive on $\R^N$.
\end{lem}

\begin{proof}
We show just (ii) implies (i).  
Suppose  $u\in \D$ is a solution to \eqref{LS2}.  
We take $\varphi=u^-$ in \eqref{eq:2.6}.  Since $f(u)u^-=0$ and $h(u)u^-=-h(u^-)u^-$,
$$\int_{\R^N}(\abs{\nabla u^-}^2+V(x)(u^-)^2+h(u^-)u^-)=0.$$
So $u^-=0$ and $u\geq 0$. 
Moreover, if $u\neq 0$, then $u>0$ by the strong maximum principle.
\end{proof}

\noindent
To study the solution to \eqref{LS2}, we consider the energy functional 
$J_\infty: E_\infty\rightarrow \R\cup\{+\infty\}$ associated to \eqref{LS2}:
    $$
    J_\infty(u)=\half\norm u^2_{E_\infty}-\int_{\R^N}Q(x)G(u),
    $$
where we write $E_\infty=H^1(\R^N)$ and we introduce the following inner product and norm to $E_\infty$:
    \begin{eqnarray*}
	&&\inp{u,v}_{E_\infty}=\int_{\R^N}\nabla u\nabla v+ V(x)uv,\\
    &&\norm{u}_{E_\infty}=\sqrt{\inp{u,u}_{E_\infty}}
		\quad \text{for}\ u,v\in E_\infty.
	\end{eqnarray*}
To study \eqref{LS2}, we introduce the notion of derivatives and critical points of $J_\infty$.

\begin{defn}\label{defn:2.1}
\begin{enumerate}[(i)]
\item For  $u$, $v\in \D$, we define
    $$  J_\infty'(u)v\equiv\inp{u,v}_{E_\infty}-\int_{\R^N}Q(x)g(u)v.
    $$
We note that $\int_{\R^N} Q(x)h(u)v$ is well-defined by Lemma \ref{lem:2.2} (ii).
\item We say that $u\in E_\infty$ is a {\sl critical point} of $J_\infty$ if $u\in\D$ and 
$J_\infty'(u)v=0$ for all $v\in\D$.
We also say that $c\in\R$ is a {\sl critical value} for $J_\infty$ if there exists a critical point 
$u\in E_\infty$ such that  $J_\infty(u)=c$.  We also use notation
    $$  \K_\infty=\{u\in\D;\, J_\infty'(u)v=0 \ \text{for all}\ v\in \D\}.
    $$
\end{enumerate}
\end{defn}

\begin{remark}
By Lemma \ref{lem:2.2} (iii), $u\in\D$ is a critical point of $J_\infty$ if and only if
    $$  J_\infty'(u)\varphi=0 \quad \text{for all}\ \varphi\in {C}_0^\infty(\R^N).
    $$
\end{remark}

%%%%%%%%%%%%%

\begin{remark}\label{remark:2.2}
There exists $\rho_\infty>0$ such that
    $$	\norm w_{E_\infty} \geq \rho_\infty \quad \text{for all}\ w\in\K_\infty\setminus\{ 0\}.
    $$
In fact, we have for small $\rho_\infty>0$
	$$	J_\infty'(u)u \geq \frac 1 4 \norm u_{E_\infty}^2 \quad \text{for all}\ u\in \D 
		\ \text{with}\ \norm u_{E_\infty}\leq \rho_\infty.
	$$
\end{remark}

\noindent
At the end of this section, we give a proof of Lemma \ref{lem:2.1}

\begin{proof}[Proof of Lemma \ref{lem:2.1}]
We can easily show (i) and (ii).  Here we give a proof to (iii).
(iii-a) By the concavity of $s\mapsto h(s);\, [0,\infty)\to \R$, we have \eqref{eq:2.2}.
Integrating \eqref{eq:2.2}, we get \eqref{eq:2.3}.  We can also verify \eqref{eq:2.4} easily.\\
(iii-b) It suffices to show \eqref{eq:2.5} for $s$, $t>0$ and $\theta\in (0,1]$.  
Setting $H^*(y)=\sup_{z\in\R} (yz-H(z))$, i.e., the convex conjugate of $H(x)$, we
have the following Fenchel's inequality:
    $$  xy \leq H(x)+H^*(y) \quad \text{for all}\ x, y\in\R.
    $$
Setting $x=t$, $y=\theta h(s)$, we have
    $$  \theta h(s)t \leq H^*(\theta h(s)) + H(t).
    $$
To show \eqref{eq:2.5}, it suffices to show 
    \be\label{eq:2.7}
    H^*(\theta h(s)) \leq \theta^2 H(s).
    \ee
We set $\phi(z)=\theta h(s)z-H(z)$.  Then we have $H^*(\theta h(s))=\sup_{z\in \R}\phi(z)$.
Noting by \eqref{eq:2.2}
    $$  \phi'(0)=\theta h(s)\geq 0,\quad \phi'(\theta s)=\theta h(s) -h(\theta s) \leq 0,
    $$
$\phi(z)$ takes its maximum in $[0,\theta s]$ and there exists $a\in [0,1]$ such that
    \begin{eqnarray*}
    H^*(\theta h(s)) &=& \phi(a\theta s) = a\theta^2 h(s)s -H(a\theta s) \\
    &\leq& 2a\theta^2 H(s)-(a\theta)^2 H(s) = (2a-a^2)\theta^2 H(s)\\
    &\leq& \theta^2 H(s).
    \end{eqnarray*}
Here we used \eqref{eq:2.3} and \eqref{eq:2.4}.  Thus we obtain \eqref{eq:2.7}.
\end{proof}

%%%%%%%%%%%%%%%%%%%%%%%%%%

\subsection{$2L$-periodic problems}
As stated in Introduction, we construct a solution of \eqref{LS2} through $2L$-periodic problems.

For $L\in\mathbb N$, we set $D_L=[-L,L]^N$ and we consider a Hilbert space
    $$  E_L=\{ u(x)\in  H^1_{loc}(\R^N);\, u(x+2Ln)=u(x)\ \mbox{for all}\ n\in\Z^N\}
    $$ 
with an inner product and a norm:
    \begin{eqnarray*}
    &&\inp{u,v}_{E_L}=\int_{D_L}\nabla u\nabla v+ V(x)uv, \\
    &&\norm u_{E_L}=\sqrt{\inp{u,u}_{E_L}}\quad \text{for}\ u,\, v\in E_L.
    \end{eqnarray*}
We denote the dual space of $E_L$ and its norm by $E_L^*$ and $\norm{\cdot}_{E_L^*}$.\\
For a subset $U\subset D_L$, we also denote
    $$  \norm u_{E_L(U)} =\left(\int_{U}\abs{\nabla u}^2+V(x)u^2\right)^\half
        \quad \text{for}\ u\in E_L.
    $$
For $u\in E_L$ we define $J_L(u)$ by
    $$  J_L(u)=\half\norm u^2_{E_L}-\int_{D_L}Q(x)G(u).
    $$
The equation corresponding to $J_L(u)$ is \eqref{PL}.
We note that we are essentially working in a bounded domain $D_L$ and $\int_{D_L} Q(x)H(u)$ is
well-defined on $E_L$.  We have

\begin{lem}
\begin{enumerate}[(i)]
\item $E_L\to \R;\, u\mapsto \int_{D_L} Q(x)H(u)$ is of class $C^1$;
\item $J_L(u)\in C^1(E_L,\R)$ and any critical point of $J_L(u)$ is a weak solution of \eqref{PL},
that is,
    $$  \int_{D_L} \nabla u\nabla \varphi+V(x)u\varphi-Q(x)g(u)\varphi=0
        \quad \text{for all}\ \varphi\in E_L,
    $$
\end{enumerate}
\end{lem}

\noindent
We denote the critical point set of $J_L$ by
    $$  \K_L=\{u\in E_L; J_L'(u)=0\}.
    $$
As in Lemma \ref{lem:2.3}, we have

\begin{lem}
Any solution to \eqref{PL}, i.e. any critical point of $J_L(u)$, is nonnegative.
\end{lem}

\noindent
As a fundamental property of $J_L(u)$, we see $J_L(u)$ possesses a mountain pass geometry  
uniformly in $L\in\N$.

\begin{lem}\label{lem:2.6}
\begin{enumerate}[(i)]
\item There are constants $C_1$, $r_1>0$ independent of $L$ such that
    $$  J_L(u)\geq C_1 \qquad \text{for}  \ \norm u_{E_L}=r_1;
    $$
\item there exists $u_0\in H_0^1(D_1)$ such that
    $$  \norm{\Phi_L({u_0})}_{E_L}>r_1 \quad \text{and}\quad J_L(\Phi_L({u_0}))<0.  
    $$
where $r_1$ is the number in (i).  Here we regard $u_0(x)$ as an element in $H_0^1(D_L)$ 
and $\Phi_L({u_0})\in E_L$ is a $2L$-periodic extension of $u_0(x)$.
In particular, there exists a constant $C_2>0$ independent of $L\in\N$ such that
    $$  \max_{t\in [0,1]} J_L(t\Phi_L({u_0})) \leq C_2.
    $$

\end{enumerate}
\end{lem}

\begin{proof}       
(i) 
Since $H(u)\geq 0$,
    $$  J_L(u)\geq \half\norm u_{E_L}^2-C_p\norm{Q}_{L^\infty}\int_{D_L}\abs u^p 
        \geq \half\norm u_{E_L}^2-C_p'\norm u_{E_L}^p.
    $$
Thus the conclusion (i) holds for small $r_1>0$ and for a suitable constant $C_1>0$.
(ii)  We choose  $u_1\in H_0^1(D_1)\setminus\{0\}$ with $u_1(x)\geq0$ in $D_1$ and consider 
$J_L(t\Phi_L({u_1}))$ for $t>0$.  Here $\Phi_L({u_1})$ is a $2L$-periodic extension of $u_1(x)$.
    \begin{eqnarray*}
    J_L(t\Phi_L({u_1}))&=&\half t^2\norm {\Phi_L({u_1})}^2_{E_L}+\int_{D_L}Q(x)G(t\Phi_L({u_1}))\\
    &=&\half t^2\norm {\Phi_L({u_1})}^2_{E_L}+\half\int_{D_L}Q(x)\bigl((tu_1)^2\log(tu_1)^2-(tu_1)^2\bigr) \\
    &=&\half t^2\Big\{\norm {u_1}_{H_0^1(D_1)}^2+\int_{D_1}Q(x)u_1^2-\int_{D_1}Q(x)u_1^2\log u_1^2\Big\}\\
    &&      -\half t^2\log t^2\int_{D_1}Q(x)u_1^2.
    \end{eqnarray*}
Since $J_L(t\Phi_L({u_1}))\to -\infty$ as $t\to\infty$, we can choose a $t_1>0$ such that $J_L(t_1\Phi_L({u_1}))<0$.
Thus for $u_0=t_1u_1$, the conclusion (ii) holds.
\end{proof}

%%%%%%%%%%%%%%%%%%%%% 

\noindent
The following property of $J_L(u)$ is based on a special feature of our nonlinearity:
    \begin{equation}\label{eq:2.8}
    G(s)-\half g(s)s = -\half s^2 \quad \text{for}\ s\geq 0.
    \end{equation}
It is useful to check Palais-Smale condition and concentration-compactness type result for $J_L(u)$.

\begin{lem}\label{lem:2.7}
For $\delta>0$ there exists a constant $C_3>0$ independent of $L\in\N$ such that 
for any $L\in\N$ and $u\in E_L$ with
    \begin{equation}\label{eq:2.9}
	\norm{J_L'(u)}_{E_L^*} \leq \delta,
    \end{equation}
we have
\begin{enumerate}[(i)]
\item $\norm{u^-}_{E_L} \leq \delta$, $\int_{D_L} H(u^-)$, $\int_{D_L} h(u^-)u^-\leq C_3\delta^2$;
\item Moreover, assume that for $M>0$
	\begin{equation}\label{eq:2.10}	
	J_L(u)\leq M.
	\end{equation}
Then there exists a constant $C_4(M,\delta)>0$ independent of $L\in\N$ such that 
$\norm u_{E_L}$, $\int_{D_L} H(u)$, $\int_{D_L} h(u)u \leq C_4(M,\delta)$.
\item For a critical point $u\in E_L$ of $J_L(u)$, we have
    \begin{equation} \label{eq:2.11}
    J_L(u)=\half \int_{D_L} Q(x)u^2.
    \end{equation}
\end{enumerate}
\end{lem}

\begin{proof}
Suppose that $L\in\N$ and $u\in E_L$ satisfies \eqref{eq:2.9}.  We note that $g(u)u^-=h(u^-)u^-$.
It follows from $\abs{J_L'(u)u^-}\leq \delta \norm{u^-}_{E_L}$ that
    $$  \norm{u^-}_{E_L}^2 + \int_{D_L} Q(x)h(u^-)u^- \leq \delta \norm{u^-}_{E_L}.
    $$
We can easily get (i) from Lemma \ref{lem:2.1} (iii-a).

To show (ii), we note by \eqref{eq:2.4} and \eqref{eq:2.8} that
    \begin{eqnarray*}
    J_L(u) -\half J_L'(u)u &=& -\int_{D_L} Q(x)\bigl( G(u)-\half g(u)u \bigr) \\
    &=& \half\int_{D_L} Q(x)(u^+)^2 + \int_{D_L} Q(x) \bigl( H(u^-)-\half h(u^-)u^- \bigr) \\
    &\geq& \half\int_{D_L} Q(x)(u^+)^2.
    \end{eqnarray*}
It follows from \eqref{eq:2.9}--\eqref{eq:2.10} that
	$$	\half\int_{D_L} Q(x)(u^+)^2 \leq M +\frac{\delta} 2\norm u_{E_L}.
	$$
Thus, by (i),
    $$  \norm u_{L^2(D_L)} \leq C_5(M,\delta)(1+\norm u_{E_L})^{\half},
    $$
where $C_5(M,\delta)>0$ is independent of $L\in\N$.  By Lemma \ref{lem:2.1} and 
Gagliard-Nirenberg inequality, we have
    \begin{eqnarray*}
    \int_{D_L} Q(x) F(u) &\leq& C_p\norm Q_{L^\infty} \norm u_{L^p(D_L)}^p 
    \leq C_p'\norm u_{E_L}^{\theta p}\norm u_{L^2(D_L)}^{(1-\theta)p} \\
	&\leq& C_p'C_5(M,\delta)^{(1-\theta)p}\norm u_{E_L}^{\theta p}(1+\norm u_{E_L})^{(1-\theta)p/2},
    \end{eqnarray*}
where $\theta\in (0,1)$ satisfies
    $$  \frac 1 p =(\half-\frac 1 N)\theta+\half(1-\theta), \quad \text{that is,} \quad
        \theta p=\frac{N(p-2)}2.
    $$
Thus it follows from $J_L(u)\leq M$ that
    $$  \half\norm u_{E_L}^2 +\int_{D_L} Q(x)H(u) 
		- C_p'C_5(M,\delta)^{(1-\theta)p}\norm u_{E_L}^{\theta p}(1+\norm u_{E_L})^{(1-\theta)p/2}
		\leq M.
    $$
Choosing $p$ close to $2$ so that $\theta p+(1-\theta)p/2<2$, we can see that (ii) holds for a suitable constant
$C_4(M,\delta)>0$ independent of $L\in\N$.

\eqref{eq:2.11} follows from \eqref{eq:2.8}.

\end{proof}

%%%%%%%%%%%%%%%%%%

\subsection{Some notation}
At the end of this section, we give some notation which will be used repeatedly
in the following sections.  

We denote 
    $$  2^*=\begin{cases}
        \frac{2N}{N-2}, &\mbox{if}\ N\geq 3, \\
        \infty,  &\mbox{if}\ N=1,2.
        \end{cases}
    $$
We will use the following subsets of $\R^N$ frequently.
    \begin{eqnarray*}
    D_1(n) &=& \{x\in\R^N;x-n\in D_1\} \quad \text{for}\ n\in\Z^N, \\
    B_r(y) &=& \{x\in\R^N;\abs{x-y}<r\} \quad \text{for}\ r>0 \ \text{and}\ y\in\R^N.
    \end{eqnarray*}
We also denote
    \begin{eqnarray*}
    B^{(L)}_r(u) &=& \{v\in E_L;\norm{u-v}_{E_L}<r\}\quad  \text{for}\ u\in E_L,\ r>0,\\
    B^{(\infty)}_r(u) &=& \{v\in E_\infty;\norm{u-v}_{E_\infty}<r\} \quad \text{for}\
                    u\in E_\infty, \ r>0.
    \end{eqnarray*}
\noindent
For $c\in \R$ and $L\in\N$ we denote 
    \begin{eqnarray*}
    {[}J_L\leq c{]}_L &=& \{u\in E_L;\, J_L(u)\leq c\},\\
    %{[}c_1 \leq J_L \leq c_2{]}_L &=& \{u\in E_L;\, c_1\leq J_L(u)\leq c_2\}, \\
    {[}J_\infty\leq c{]}_\infty &=& \{u\in \D;\, J_\infty(u)\leq c\}.
    %{[}c_1 \leq J_\infty \leq c_2{]}_\infty&=&\{u\in \D;\, c_1\leq J_\infty(u)\leq c_2\}.
    \end{eqnarray*}
In a similar way, we use ${[}J_L\geq c{]}_L$, ${[}J_\infty\geq c{]}_\infty$ etc.\\
In what follows, for $u\in H_0^1(D_L)$, we denote its $2L$-periodic extension by $\Phi_L(u)\in E_L$.
We choose and fix a function $\psi(x)\in {C}_0^\infty(\mathbb{R}^N)$ such that
    \begin{eqnarray*}
        \psi(x)&=&\begin{cases}
        1\ &\mbox{for}\ |x|\leq \frac{1}{4},\\ 
        0\ &\mbox{for}\ | x|\geq \half,
        \end{cases}\\
        |{\nabla \psi(x)}| &\leq& 8 \ \hbox{for all}\ x\in\mathbb{R}^N.
    \end{eqnarray*}
For $s>0$, we set $\psi_s(x)=\psi({x\over s})$.
We note that for $0<R\leq L$, $u\in H_{loc}^1(\mathbb{R}^N)$, 
    $$  (\psi_Ru)(x) \equiv \psi_R(x)u(x) \in H^1_0(D_L).
    $$
Thus $\Phi_L({\psi_Ru})(x)$, i.e., $2L$-periodic extension of $\psi_Ru$, 
is an element of $E_L$.

%%%%%%%%%%%%%%%%%%

\setcounter{equation}{0}
\section{Concentration-compactness type result}\label{section:3}
In this section we give a concentration-compactness type result.
It will play an important role when we take a limit as $L\to\infty$.

To state our concentration-compactness type result, we need notation:
    $$  \dist_L(y,y') = \min_{n\in \Z^N} \abs{y-y'-2Ln}.
    $$
$\dist_L(y,y')$ is a distance in $\R^N/\sim_L$, where an equivalence relation
$\sim_L$ is given by
    $$  y\sim_L y' \quad \text{if and only if}\quad
        y-y'=2Ln \ \text{for some}\ n\in\Z^N.
    $$
Our concentration-compactness type result is the following

\begin{pro}\label{pro:3.1} 
Assume that $(L_j)_{j=1}^\infty\subset\N$ and $u_j\in E_{L_j}$ ($j=1,2,\cdots$) 
satisfy for some $c>0$
    \be\label{eq:3.1}
        L_j\to\infty, \quad J_{L_j}(u_{j})\to c>0, \quad 
        \norm{J_{L_j}^\prime(u_j)}_{(E_{L_j})^*}\to 0 \quad \mbox{as}\ j\rightarrow \infty. 
    \ee  
Then 
\begin{enumerate}[(i)]
\item $\norm{u_j}_{E_{L_j}}$, $\int_{D_{L_j}}h(u_j)u_j$, $\int_{D_{L_j}}H(u_j)$ are bounded;
\item There exists $m\in \N$, $(w^\ell)_{\ell=1}^m \subset \K_\infty\setminus\{0\}$ and 
subsequence $(j_k)_{k=1}^\infty$ and sequences $(y^\ell_{j_k})_{k=1}^\infty\subset\R^N$ with 
$y_{j_k}^\ell\in D_{L_{j_k}}$ $(\ell=1,2,\cdots,m)$ such that 
\begin{enumerate}[(a)]
    \item For $\ell\neq\ell'$,
        \be\label{eq:3.2} 
        \dist_{L_{j_k}}(y_{j_k}^\ell,y_{j_k}^{\ell'})\to\infty \quad \text{as}\ k\to\infty.
        \ee
    \item For any $R_{j_k}>0$ with $R_{j_k} \leq L_{j_k}$ and $R_{j_k}\to \infty$ as $k\to\infty$
        \begin{equation}\label{eq:3.3}\\
            \norm{u_{j_k}-\sum_{\ell=1}^m\Phi_L({\psi_{R_{j_k}}w^\ell})(x-y_{j_k}^\ell)}_{E_{L_{j_k}}}
            \to 0 \quad \text{as} \ k\to\infty.
        \end{equation}
    \item 
        \begin{equation}\label{eq:3.4}
        c=\sum_{\ell=1}^m J_\infty(w^\ell)
        \end{equation}
    \item
        \begin{equation}\label{eq:3.5}
        \int_{D_{L_{j_k}}}H(u_{j_k})\to\sum_{\ell=1}^m\int_{\R^N}H(w^\ell)
        \end{equation}
\end{enumerate}
\end{enumerate}
\end{pro}

\noindent 
To show Proposition \ref{pro:3.1}, for a sequence $u_j\in E_{L_j}$ ($j=1,2,\cdots$) satisfying
\eqref{eq:3.1} we need to give estimates of $\int_{\R^N} H(u_j)$ as well as $\norm{u_j}_{E_{L_j}}$.

First we show

\begin{lem}\label{lem:3.2}
For any $q\in (2,2^*)$ there exists a constant $C_q>0$ independent of $L\in\N$ such that 
for $L\in\N$, $u\in E_L$, it holds that
    \begin{equation}\label{eq:3.6}
    \norm u_{L^q(D_L)}^q\leq C\left(\sup_{n\in\Z^N}\norm u_{L^q(D_1(n))}\right)^{q-2}\norm u_{E_L}^2.
    \end{equation}
\end{lem}

\begin{proof}
For $u\in E_L$, $n\in\Z^N$, $q\in(2,2^*)$, by Sobolev inequality
we have
    \begin{equation*}
    \begin{aligned}
    \norm u_{L^q(D_1(n))}^q&= \norm u_{L^q(D_1(n))}^{q-2}\norm u_{L^q(D_1(n))}^2\\
    &\leq C \left(\sup_{n\in\Z^N}\norm u_{L^q(D_1(n))}\right)^{q-2}
        \norm u_{E_L(D_1(n))}^2.
    \end{aligned}
    \end{equation*}
Summing up for $n\in \{-L,-L+1,\cdots,L-1\}^N$, we get \eqref{eq:3.6}. 
\end{proof}

\noindent
Next we show

\begin{lem}\label{lem:3.3}
Let $q\in (2,2^*)$ and suppose that $L_j$, $v_j$ ($j=1,2,\cdots$) satisfy $L_j\to\infty$, 
$v_j\in E_{L_j}$ and 
    $$  \norm{J_{L_j}'(v_j)}_{(E_{L_j})^*}\to 0 \quad \text{and}\quad
        \sup_{j\in\N} \norm{v_j}_{E_{L_j}} <\infty.
    $$
Moreover suppose that
    $$  \sup_{n\in\N} \norm{v_j}_{L^q(D_1(n))} \to 0.
    $$
Then
    $$  \norm{v_j}_{E_{L_j}}\to 0, \quad \int_{D_{L_j}} H(v_j)\to 0, \quad
        J_{L_j}(v_j)\to 0.
    $$
\end{lem}

\begin{proof}
By Lemma \ref{lem:3.2}, we have
    $$  \norm{v_j}_{L^q(D_{L_j})} 
        \leq C\left(\sup_{n\in\Z^N} \norm{v_j}_{L^q(D_1(n))}\right)^{q-2} \norm{v_j}_{E_{L_j}}^2
        \to 0,
    $$
which implies $\displaystyle\int_{D_{L_j}}Q(x)f(v_j)v_j\to 0$.  Thus
    \begin{eqnarray*}
    &&\norm{v_j}_{E_{L_j}}^2 +\int_{D_{L_j}} Q(x)H(v_j) 
    \leq \norm{v_j}_{E_{L_j}}^2 +\int_{D_{L_j}} Q(x)h(v_j)v_j \\
    &&= J_{E_{L_j}}'(v_j)v_j + \int_{D_{L_j}} Q(x)f(v_j)v_j \to 0.
    \end{eqnarray*}
Therefore we have the conclusion of Lemma \ref{lem:3.3}.
\end{proof}

\noindent
Proof of Proposition \ref{pro:3.1} is rather lengthy and we will give a proof in Appendix.

%%%%%%%%%%%%%%%%%%%%%

\setcounter{equation}{0}
\section{One bump solutions}\label{section:4}
As we observed in Lemma \ref{lem:2.6}, $J_L(u)$ has a mountain pass geometry uniformly in $L\in\N$.
We define mountain pass values for $J_L(u)$ by
    \begin{eqnarray*}
    b_L&=&\inf_{\gamma\in\Gamma_L}\max_{\tau\in[0,1]}J_L(\gamma(\tau)),  \\
    \Gamma_L&=&\{\gamma\in C([0,1],E_L); \gamma(0)=0, J_L(\gamma(1))<0)\}.
    \end{eqnarray*}
Our main result in this section is the following theorem.

\begin{thm} \label{thm:3} 
Assume (A1) and (A2).  Then we have
\begin{enumerate}[(i)]
\item For any $L\in\N$, $b_L$ is attained by a positive solution $u_L(x)\in E_L$ of \eqref{PL}.
\item $b_L$ is characterized as
    $$  b_L=\inf_{u\in {\cal N}_L} J_L(u),
    $$
where
    $$  {\cal N}_L = \{ u\in E_L\setminus\{0\};\, J_L'(u)u=0\}.
    $$
\item $b_L\to b_\infty$ as $L\to\infty$, where $b_\infty$ is defined in \eqref{eq:1.2}.
\end{enumerate}
\end{thm}

%%%%

\noindent
As a corollary to Theorem \ref{thm:3}, we have

\begin{corollary}[c.f. Theorem 1.2 of \cite{Sz}]\label{corollary:4.1}
Assume (A1) and (A2).  Let $b_\infty$ be a number defined in \eqref{eq:1.2}.
Then $b_\infty\in(0,\infty)$ and $b_\infty$ is achieved by a positive solution of \eqref{LS1}. 
\end{corollary}   

%%%%

In this section, we will first prove that $b_L$ is a critical value of $J_L$, and then use 
the concentration compactness (Proposition \ref{pro:3.1})
to get a nonzero critical point for $J_\infty$. Then we will prove $\displaystyle\lim_{L\to\infty}b_L=b_\infty$.
By Lemma \ref{lem:2.6}, we know for every $L\in \N$, $J_L$ has mountain-pass geometry.
To prove $b_L$ is a critical value of $J_L$, we only need to prove that $J_L$ satisfies the (PS) condition. 
The first statement (i) in Theorem \ref{thm:3} follows from mountain pass theorem using the following 
lemma.

\begin{lem}\label{lem:4.2}
For any $L\in\N$, $J_L(u)$ satisfies the Palais-Smale condition.  More precisely, if a sequence
$(u_k)_{k=1}^\infty\subset E_L$ satisfies
    \begin{equation}\label{eq:4.1}
    (J_L(u_k))_{k=1}^\infty\quad \mbox{is bounded and} \quad 
    \norm{J_L'(u_k)}_{E_L^*}\to 0,
    \end{equation}
then $(u_k)_{k=1}^\infty$ has a convergent subsequence in $E_L$.
\end{lem}

\begin{proof}
Suppose that $(u_k)_{k=1}^\infty\subset E_L$ satisfies \eqref{eq:4.1}.  By Lemma \ref{lem:2.7},
we can deduce that $(u_k)_{k=1}^\infty$ is a bounded sequence in $E_L$.
Thus, extracting a subsequence if necessary,  we may assume that for some $u_0\in E_L$
     $$ u_k\rightharpoonup u_0\  \mbox{weakly in}\ E_L.
     $$
Since we are working in a bounded domain $D_L$ essentially, we can prove $u_k\to u_0$ 
strongly in $E_L$ in a standard way.
\end{proof}

\noindent
To show the second statement (ii) in Theorem \ref{thm:3}, we need the following

%%%%%%%%%%%%%

\begin{lem}\label{lem:4.3}
\begin{enumerate}[(i)]
\item For $L\in\N$ and for any $u\in E_L$ with $u^+\not=0$, we have

\begin{enumerate}[(a)] 
\item $\{ tu;\, t\geq 0\}\cap {\cal N}_L\not=\emptyset$ and there exists a unique $t_u>0$ such
that $t_uu\in{\cal N}_L$.
\item $J(tu)<J(t_uu)$ for all $t\in (0,\infty)\setminus\{ t_u\}$.
\end{enumerate}
\item For any $u\in {\cal N}_L$, we have
    $$  \sup_{t\in [0,\infty)} J_L(tu) \leq J_L(u).
    $$
In particular, there exists $\gamma_u(t)\in \Gamma_L$ such that $\max_{t\in [0,1]}J(\gamma_u(t)) \leq J_L(u)$.
\end{enumerate}
\end{lem}

\begin{proof}
For $u\in E_L$ with $u^+\not=0$, we set $\phi(t)=J_L(tu)$ for $t\in [0,\infty)$.  Then 
    \begin{eqnarray*}
    \phi'(t) &=& t\norm u_{E_L}^2 -\int_{D_L} Q(x)g(tu)u \\
    &=& t\left( \norm u_{E_L}^2 -\int_{D_L} Q(x)\frac{g(tu)} t u\right).
    \end{eqnarray*}
Thus by Lemma \ref{lem:2.1}, $\frac 1 t\phi'(t)$ is a strictly decreasing function of $t\in (0,\infty)$.
Since $u^+\not=0$, we can also easily see that $\phi'(t)<0$ for large $t>0$ and $\phi'(t)>0$ for small $t>0$.
Thus there exists a unique $t_u\in (0,\infty)$ such that $\phi'(t_u)=0$ and
    \begin{eqnarray*}
    \phi'(t) &>& 0 \qquad \text{for}\ t\in (0,t_u),\\
    \phi'(t) &<& 0 \qquad \text{for}\ t\in (t_u,\infty).
    \end{eqnarray*}
Noting $\{ tu;\, t>0\}\cap {\cal N}_L =\{ tu;\, \phi'(t)=0\}$, we have (i).  

Noting also $u^+\not=0$ for all $u\in {\cal N}_L$ and setting $\gamma_u(t)=tMu$ for a large constant $M\gg 1$, 
(ii) follows from (i).  
\end{proof}

\noindent
In a similar way to Lemma \ref{lem:4.3}, we have

\begin{lem}\label{lem:4.4}
For any $u\in {\cal N}_\infty$, 
    $$  \sup_{t\in [0,\infty)} J_\infty(tu) \leq J_\infty(u).
    $$
\end{lem}

\begin{proof}[Proof of (ii) of Theorem \ref{thm:3}]
By lemma \ref{lem:4.3}, for any $u\in {\cal N}_L$ there exists a path $\gamma_u\in \Gamma_L$ such that 
$\max_{t\in [0,1]} J_L(\gamma_u(t)) \leq J_L(u)$.  Thus we have
    $$  b_L \leq \max_{t\in [0,1]} J_L(\gamma_u(t)) \leq J_L(u).
    $$
Since $u\in {\cal N}_L$ is arbitrary, we have
    \begin{equation}\label{eq:4.2}
        b_L \leq \inf_{u\in {\cal N}_L} J_L(u).
    \end{equation}
On the other hand, since $J_L(u^+) -\half J_L'(u^+)u^+ = \half\int_{D_L}Q(x)(u^+)^2\geq 0$ 
by \eqref{eq:2.8},  we have 
for any $\gamma\in\Gamma_L$
    $$  \half J_L'(\gamma(1)^+)\gamma(1)^+ \leq J_L(\gamma(1)^+) \leq J_L(\gamma(1))<0.
    $$
Thus there exists $t_0\in (0,1)$ such that $J_L'(\gamma(t_0)^+)\gamma(t_0)^+=0$, 
that is, $\gamma(t_0)^+\in {\cal N}_L$.  Thus, 
    $$  \inf_{v\in{\cal N}_L} J_L(v) \leq  J_L(\gamma(t_0)^+) \leq \max_{t\in [0,1]} J_L(\gamma(t)^+) 
        \leq \max_{t\in [0,1]} J_L(\gamma(t)).
    $$
Since $\gamma\in\Gamma_L$ is arbitrary, we have $b_L\geq \inf_{v\in{\cal N}_L} J_L(v)$.  Together with
\eqref{eq:4.2}, we have (ii) of Theorem \ref{thm:3}.
\end{proof}

\noindent
Finally we show (iii) of Theorem \ref{thm:3} and Corollary \ref{corollary:4.1}.

\begin{proof}[Proof of (iii) of Theorem \ref{thm:3} and Corollary \ref{corollary:4.1}]
First we remark that there are constants $C_1$, $C_2>0$ such that
    $$  b_L\in [C_1,C_2] \quad \text{for all}\ L\in \N,
    $$
which follows from Lemma \ref{lem:2.6}.

For any $\varepsilon>0$, we can find $u_0\in\mathcal{N}_\infty$ such that $J_\infty(u_0)<b_\infty+\varepsilon$. 
By Lemma \ref{lem:4.4}, we have
    $$   \max_{t\in[0,\infty)}J_\infty(tu_0)
        =J_\infty(u_0)<b_\infty+\varepsilon.
    $$
Considering  $\psi_Lu_0\in H^1_0(D_L)$ for $L\in\N$, we have
    $$  \max_{t\in [0,\infty)}\int_{D_L} \frac {t^2} 2\abs{\nabla (\psi_Lu_0)}^2 -Q(x)G(t\psi_Lu_0)
        \to \max_{t\in [0,\infty)} J_\infty(tu_0) \quad \text{as}\ L\to\infty.
    $$
That is,
    $$  \max_{t\in [0,\infty)} J_L(t\Phi_L({\psi_Lu_0})) \to \max_{t\in[0,\infty)}J_\infty(tu_0)
        <b_\infty+\varepsilon  \quad \text{as}\ L\to\infty,
    $$
from which we have
    $$  \limsup_{L\to\infty} b_L \leq \limsup_{L\to\infty} \max_{t\in [0,\infty)}J_L(t\Phi_L({\psi_Lu_0})) 
        \leq b_\infty+\varepsilon.
    $$
Thus we have $\limsup_{L\to\infty} b_L \leq b_\infty$.

On the other hand, we assume that $b_{L_k}\to c\equiv \liminf_{L\to\infty} b_L\in [C_1,b_\infty]$ for some
subsequence $L_k\to \infty$.  By Proposition \ref{pro:3.1}, we have
    $$  c=\sum_{\ell=1}^m J_\infty(w^\ell)
    $$
for some $m\geq 1$ and $w^\ell\in \K^\infty\setminus\{ 0\}\subset {\cal N}_\infty$.  Thus we have
$c\geq b_\infty=\inf_{u\in{\cal N}_\infty} J_\infty(u)$.   Therefore we have $b_L\to b_\infty$ as
$L\to \infty$ and there exists a $w\in \K_\infty\setminus\{ 0\}\subset {\cal N}_\infty$ such that
    $$  b_\infty = J_\infty(w).
    $$
This completes the proof of (iii) of Theorem \ref{thm:3} and Corollary \ref{corollary:4.1}.
\end{proof}

%%%%%%%%%%%%%%%%%%%%%%%%%%%%%%%%%%

\setcounter{equation}{0}
\section{Multi-bump solutions}\label{section:5}
In this section, we assume that there exists an $\alpha\in(0,\frac{1}{10}b_\infty)$ such that
    \begin{equation} \tag{\textasteriskcentered}
    (\K_\infty\cap[J_\infty\leq b_\infty+\alpha]_\infty)/\Z^N\ \mbox{is finite}.
    \end{equation}
Under the assumption \eqref{*}, choosing $\alpha>0$ smaller if necessary, we may also assume that
    $$  \K_\infty\cap[J_\infty=b_\infty]_\infty = \K_\infty\cap[0<J_\infty\leq b_\infty+\alpha]_\infty.
    $$
We choose and fix 
    $$  \omega,\, \omega'\in \F_\infty,
    $$
arbitrary and for large $R\gg 1$ we choose $P\in\Z^N$ and $L\in\N$ such that
    \be\label{sharp}\tag{$\sharp$}
    5R \leq \abs P \quad \text{and}\quad 2\abs P \leq L.
    \ee
We try to find a critical point in a neighborhood of 
    $$  \wwP = \Phi_L({\psi_R \omega})(\cdot)+\Phi_L({\psi_R \omega'})(\cdot-P).
    $$
In what follows, we always assume that $R$, $P$, $L$ satisfy \eqref{sharp}.

Later for a fixed $P\in\Z^N$ with $\abs P > 5R$, we take a limit as $L\to\infty$ to obtain our main
Theorem \ref{thm:2}.

%%%%%

\subsection{A gradient estimate and deformation argument}
To show the existence of a critical point in a neighborhood of $\wwP$, estimates of $J_L'(u)$ in 
annular neighborhoods of $\wwP$ are important.
We need the following notation to state our estimates.  

We note that under \eqref{sharp}
    $$  B_{2R}(0), \, B_{2R}(P) \subset D_L, \quad B_{2R}(0)\cap B_{2R}(P)=\emptyset.
    $$
For $R\gg 1$, $P\in\Z^N$, $L\in\N$ satisfying \eqref{sharp}, we set for $t\in[\frac R 2,2R]$,
    $$  \A{t} = D_L\setminus(B_t(0)\cup B_t(P)).
    $$
We also use notation for $t\in[\frac R 2,2R]$ and $u\in E_L$
    $$  \norm u_{E_L(\A{t})}^2 = \int_{\A{t}} \abs{\nabla u}^2 +V(x)u^2
    $$
and we define $J_{\A{t}}(u):\, E_L\to\R$ by
    $$  J_{\A{t}}(u)=\half\norm{u}_{E_L(\A{t})}^2 -\int_{\A{t}}Q(x)G(u).
    $$
We also denote for $c\in\R$
    $$  [J_{\A{t}}\leq c]_L = \{ u\in E_L;\, J_{\A{t}}(u) \leq c\} \ \text{etc.}
    $$
First we need the following lemma.

\begin{lem}\label{lem:5.1}
There exists $r_0>0$ and $R_0\geq 1$ with the following properties:
\begin{enumerate}[(I)]
\item For any $L\in \N$
    $$  J_L(u),\, J_L'(u)u \geq \frac 1 4\norm u_{E_L}^2 \quad
        \text{for all}\ \norm u_{E_L} \leq 2r_0.
    $$
In particular,
    $$  \norm u_{E_L} \geq 2r_0 \quad \text{for all}\ u\in\K_L\setminus\{ 0\}.
    $$
\item $\norm u_{E_\infty} \geq 2r_0$ for all $u\in \K_\infty\setminus\{ 0\}$;
\item For any $R$, $P$, $L$ with $R\geq 1$ and \eqref{sharp},
    $$  J_{\A{R}}(u):\, \{ u\in E_L;\, \norm u_{E_L(\A{R})} \leq r_0\}\to \R
    $$ 
is strictly convex and it holds that for $t\in [\frac R 2,2R]$
    \begin{eqnarray}
    &&J_{\A{t}}(u) \geq \frac 1 4\norm u_{E_L(\A{t})}^2,   \nonumber\\
    &&\int_{\A{t}} Q(x)f(u)u,\  \int_{\A{t}} Q(x)F(u)\leq \frac 1 8 \norm u_{E_L(\A{t})}^2 \label{eq:5.1}
    \end{eqnarray}
for all $u\in E_L$ with $\norm u_{E_L(\A{t})}\leq 2r_0$. 
\item For $R$, $P$, $L$ with $R\geq R_0$ and \eqref{sharp}, we have
    \be\label{eq:5.2}
	\norm{\psi_Rv-(\psi_Rv')(\cdot-n)}_{E_\infty},\ \norm{\psi_Rv-(\psi_Rv)(\cdot-n')}_{E_\infty} > 4r_0
    \ee
for all $v$, $v'\in \F_\infty$ with $v\neq v'$ and $n\in\Z^N$, $n'\in\Z^N\setminus \{0\}$.
\item For $R$, $P$, $L$ with $R\geq R_0$ and \eqref{sharp} and for $u\in B_{r_0}^{(L)}(\wwP)$
    $$  \half\int_{D_L} Q(x)u^2 \in (2b_\infty-\alpha,2b_\infty+\alpha).
    $$
\item For any $x_0\in \R^N$, if $u(x)\in H^1(B_2(x_0))$ satisfies
	\begin{align*}
	-&\Delta u+V(x)u=Q(x) g(u) \quad \text{in}\ B_2(x_0),\\
	&\norm u_{H^1(B_2(x_0))} \leq r_0.
	\end{align*}
Then $u$ satisfies
	$$	\norm u_{L^\infty(B_1(x_0))} \leq \frac 1 e.
	$$
\end{enumerate}
\end{lem}

\begin{proof}
(I)--(III) can be checked easily for $r_0>0$ small independent of $R$, $P$, $L$.
For (IV), we set
    \begin{align*}\mu_1&=\inf\{\norm{v-v'}_{E_\infty};\, v,v'\in \K_\infty\cap[J_\infty= b_\infty],v\neq v'\}\\
    &=\inf\{\norm{v-v'(\cdot-n)}_{E_\infty},\norm{v-v(\cdot-n')};\, v,v'\in \F_\infty,\, v\neq v', \\
    &\qquad\qquad\qquad n,\, n'\in\Z^N,\, n'\neq0\} >0.  
    \end{align*}
Choosing $r_0\in (0, \frac 1 3\mu_1]$, we can see \eqref{eq:5.2} holds for large $R$.

For (V), we note that $\half\int_{\R^N}Q(x)\omega^2=\half\int_{\R^N}Q(x)\omega'^2=b_\infty$.  Since 
    $$  \half\int_{D_L} Q(x)\wwP^2 \to \half\int_{\R^N}Q(x)\omega^2 + \half\int_{\R^N}Q(x)\omega'^2
        =2b_\infty
    $$
as $R\to\infty$ ($L\geq 10R$), there exists $R_0\geq 1$ such that
    $$  \half\int_{D_L} Q(x)\wwP^2 \in (2b_\infty-\half\alpha,2b_\infty+\half\alpha)
    $$
for $R$, $P$, $L$ with $R\geq R_0$ and \eqref{sharp}.  Thus, choosing $r_0>0$ small, we have
(V).\\
Proof of (IV) is rather lengthy.  We give it in the appendix.
\end{proof}

\noindent
In what follows, we assume $r\in (0,r_0)$ without loss of generality and we try to find
a critical point of $J_L$ in $B_r^{(L)}(\wwP)$.  We will use Lemma \ref{lem:5.1} repeatedly.

Our main result in this subsection is

%%%%%%%%%%%%%%%%%%%%%%%%%%%%

\begin{pro}\label{pro:5.2}
For any $r_1,r_2\in(0,r_0]$ with $r_1<r_2$ and for any $\rho>0$, there exists 
$R_1=R_1(r_1,r_2,\rho)\geq R_0$, $\nu_1=\nu_1(r_1,r_2,\rho)>0$ such that for all $R$, $P$, $L$ with 
$R\geq R_1$ and \eqref{sharp}, if 
    $$  u\in \Big( B_{r_2}^{(L)}(\wwP)\setminus B_{r_1}^{(L)}(\wwP)\Big)
        \cup \Big( (B_{r_2}^{(L)}(\wwP) \cap [J_{\A{R}}\geq\rho]_L\Big),
    $$ 
then there exists $\varphi_u\in E_L$ with $\norm{\varphi_u}_{E_L}\leq 1$ such that
\begin{enumerate}[(a)]
\item $J_L'(u)\varphi_u\geq\nu_1$;
\item moreover if $J_{\A{R}}(u)\geq\rho$, then
    $$  J_{\A{R}}'(u)\varphi_u\geq0.
    $$
\end{enumerate}
\end{pro}

\begin{proof} 
Proof is divided into two parts.  

\smallskip

\noindent
{\sl Step 1.  
For any $0<r_1<r_2\leq r_0$ there exists $\widetilde R_1=\widetilde R_1(r_1,r_2)>0$ and 
$\widetilde{\nu_1}=\widetilde{\nu_1}(r_1,r_2)>0$ such that for any 
$R$, $P$, $L$ with $R\geq \widetilde R_1$ and \eqref{sharp}
    $$  \norm{J_L'(u)}_{E_L^*} \geq \widetilde\nu_1 \ \text{for all}\ 
        u\in B_{r_2}^{(L)}(\wwP)\setminus B_{r_1}^{(L)}(\wwP).
    $$
}

\smallskip

\noindent
In fact, if the conclusion does not hold, there exist sequences $R_j$, $P_j$, $L_j$, $u_j$ satisfying
\eqref{sharp} and $u_j\in E_{L_j}$ such that
    \begin{eqnarray}
    &&R_j\to \infty, \nonumber \\
    &&u_j\in B_{r_2}^{(L_j)}(\Omega_{R_j,P_j,L_j})\setminus B_{r_1}^{(L_j)}(\Omega_{R_j,P_j,L_j}), \label{eq:5.3}\\
    &&\norm{J_{L_j}'(u_j)}_{E_{L_j}^*} \to 0.   \nonumber
    \end{eqnarray}
Clearly $\norm{u_j}_{E_{L_j}}$ is bounded.  
By Lemma \ref{lem:2.7}, we have
	$$	\norm{u_j^-}_{E_{L_j}},\ \int_{D_{L_j}} H(u_j^-), \ \int_{D_{L_j}} h(u_j^-)u_j^- \to 0 
		\quad \hbox{as}\ j\to\infty.
	$$
Thus,
	\begin{eqnarray*}
	J_{L_j}(u_j) &=& \half J_{L_j}'(u_j)u_j -\int_{D_{L_j}} Q(x)(G(u_j)-\half g(u_j)u_j) \\
		&=& \half \int_{D_{L_j}} Q(x)u_j^2 + o(1) \quad \text{as}\ j\to\infty.
	\end{eqnarray*}
Thus by Lemma \ref{lem:5.1} (V), we have
    \be\label{eq:5.4}
    J_{L_j}(u_j)\in \left(2b_\infty-\alpha,2b_\infty+\alpha\right) \quad \text{for large}\ j.
    \ee
Applying Proposition \ref{pro:3.1}, we have for some $m\in \N$ and $(w^\ell)_{\ell=1}^m\subset\K_\infty\setminus\{ 0\}$
    $$  J_{L_j}(u_j) \to \sum_{\ell=1}^m J_\infty(w^\ell) \geq mb_\infty
    $$
after extracting a subsequence.  By \eqref{eq:5.4}, we have $m\in \{ 1,2\}$.
On the other hand, by \eqref{eq:5.3}, we have for some constant $C>0$
    $$  \norm{u_j}_{H^1(B_{R_j}(0))}, \ \norm{u_j}_{H^1(B_{R_j}(P_j))} \geq C 
        \quad \text{for all}\ j,
    $$
from which we can see $m=2$.

Thus we can find for some $w^1$, $w^2\in \F_\infty$ and $y_j^1$, $y_j^2\in \Z^N$
    $$  \norm{u_j - \sum_{\ell=1}^2 \Phi_{L_j}(\psi_{R_j}w^\ell)(\cdot-y_j^\ell)}_{E_{L_j}} \to 0.
    $$
By Lemma \ref{lem:5.1} (IV), it follows from \eqref{eq:5.3} that
$w^1=\omega$, $w^2=\omega'$, $y_j^1=0$, $y_j^2=P_j$.
Therefore
    $$ \norm{u_j-\Phi_{L_j}(\psi_{R_j}\omega)-\Phi_{L_j}(\psi_{R_j}\omega')(\cdot-P_j)}_{E_{L_j}}\to 0.
    $$
In particular, we have
    $$  u_j\in B_{r_1}^{(L_j)}(\Phi_{L_j}(\psi_{R_j}\omega)+\Phi_{L_j}(\psi_{R_j}\omega')(\cdot-P_j)) \quad
        \text{for large}\ j,
    $$
which contracts with \eqref{eq:5.3}.  Thus we have the conclusion of Step 1.

\smallskip

\noindent
{\sl Step 2: For any $\rho>0$ there exists $\widetilde R_2=\widetilde R_2(\rho)>0$ and $\widetilde \nu_2=\widetilde \nu_2(\rho)>0$
such that for any $R$, $P$, $L$ with $R\geq \widetilde R_2$ and \eqref{sharp} and for any 
    \be\label{eq:5.5}
    u\in B_{r_0}^{(L)}(\wwP) \cap [J_{\A{R}} \geq \rho]_L,
    \ee
there exists $\varphi_u\in E_L$ such that
    \begin{eqnarray}
    &&\norm{\varphi_u}_{E_L} \leq 1, \label{eq:5.6}\\
    &&J_L'(u)\varphi_u \geq \widetilde\nu_2, \label{eq:5.7}\\
    &&J_{\A{R}}'(u)\varphi_u \geq 0. \label{eq:5.8}
    \end{eqnarray}
}

\smallskip

\noindent
First we remark that \eqref{eq:5.5} implies
    \begin{eqnarray*}
    \norm u_{E_L(\A{\frac R 2})} &\leq& \norm{u-\wwP}_{E_L} + \norm{\wwP}_{E_L(\A{\frac R 2})} \\
        &\leq& r_0 + \norm{\wwP}_{E_L(\A{\frac R 2})}.
    \end{eqnarray*}
Since $\norm{\wwP}_{E_L(\A{\frac R 2})}\to 0$ as $R\to \infty$ uniformly in $P$, $L$ satisfying \eqref{sharp},
we may assume that for large $\widetilde R_2\geq 0$, \eqref{eq:5.5} implies
    \be\label{eq:5.9}
    \norm u_{E_L(\A{{\frac R 2}})} \leq 2r_0.
    \ee
Denoting the maximal integer less than $\frac R 2$ by $[R/2]$, we have
    $$  \sum_{j=1}^{[R/2]} \norm u_{E_L(\A{\frac R 2+j-1})\setminus \A{\frac R 2+j})}^2
        \leq \norm u_{E_L(\A{\frac R 2})}^2 \leq 4 r_0^2.
    $$
There exists $j_u\in \{ 1,2,\cdots, {[R/2]}\}$ such that
    \be\label{eq:5.10}
        \norm u_{E_L(\A{\frac R 2+j_u-1}\setminus \A{\frac R 2+j_u})}
        \leq \frac{2r_0}{\sqrt {[R/2]}}.
    \ee
We choose a $2L$-periodic function $\chi(x)\in C^1(\R^N)$ such that
    \begin{eqnarray*}
    && \chi(x)\in [0,1],\quad \abs{\nabla \chi(x)} \leq 2 \quad \text{for all}\ x\in\R^N, \\
    && \chi(x)=\begin{cases}
                0 &\text{for}\ x\in D_L\setminus \A{\frac R 2 +j_u-1}, \\
                1 &\text{for}\ x \in \A{\frac R 2 +j_u}.
                \end{cases}
    \end{eqnarray*}
We take $\widetilde R_2$ larger so that for $R\geq \widetilde R_2$, \eqref{eq:5.10} implies
    \be\label{eq:5.11}
    \left|\int_{\A{\frac R 2+j_u-1}\setminus \A{\frac R 2+j_u}} \nabla u\nabla(\chi u) + V(x)\chi u^2
        -Q(x)f(u)\chi u\right| \leq \frac\rho 2.
    \ee
Under the assumption $J_{\A{R}}(u)\geq \rho$, we have by \eqref{eq:5.1}%
    \begin{eqnarray}
    &&J_{\A{\frac R 2+j_u}}'(u)({\chi u}) = J_{\A{\frac R 2+j_u}}'(u)u \nonumber \\
    &=& \norm u_{E_L(\A{\frac R 2 +j_u})}^2 + \int_{\A{\frac R 2+j_u}} Q(x)h(u)u 
		- \int_{\A{\frac R 2+j_u}} Q(x)f(u)u \nonumber \\
    &\geq& \half\norm u_{E_L(\A{\frac R 2+j_u})}^2 + \int_{\A{\frac R 2+j_u}} Q(x)h(u)u \nonumber \\
    &\geq& \half\norm u_{E_L(\A{\frac R 2+j_u})}^2 + \int_{\A{\frac R 2+j_u}} Q(x)H(u) \nonumber \\
    &\geq& \half\norm u_{E_L(\A{R})}^2 + \int_{\A{R}} Q(x)H(u) \nonumber \\
    &\geq& J_{\A{R}}(u) \geq \rho.  \label{eq:5.12}
    \end{eqnarray}
It follows from \eqref{eq:5.11} and \eqref{eq:5.12} that
    \begin{eqnarray*}
    &&J_L'(u)({\chi u}) \\
    &=& J_{\A{\frac R 2+j_u}}'(u)u \\
    &+&  \int_{\A{\frac R 2+j_u-1}\setminus \A{\frac R 2+j_u}} \nabla u\nabla(\chi u) + V(x)\chi u^2
        +Q(x)h(u)\chi u
        -Q(x)f(u)\chi u \\
    &\geq& \rho + \int_{\A{\frac R 2+j_u-1}\setminus \A{\frac R 2+j_u}} \nabla u\nabla(\chi u) + V(x)\chi u^2
        -Q(x)f(u)\chi u \\
    &\geq& \half\rho.
    \end{eqnarray*}
By \eqref{eq:5.9}, clearly we have
    \begin{eqnarray*}
    \norm{\chi u}_{E_L}^2 &=& \int_{D_L} \abs{\nabla(\chi u)}^2 +V(x)\chi^2u^2 \\
    &\leq& \int_{D_L} 2\abs{\nabla \chi}^2\abs u^2 + 2\chi^2\abs{\nabla u}^2+V(x)\chi^2u^2\\
    &\leq& C\norm u_{E_L(\A{\frac R 2})}^2 \leq 4Cr_0^2.
    \end{eqnarray*}
Thus, setting $\varphi_u=\frac {\chi u} {\norm{\chi u}_{E_L}}$, we have
\eqref{eq:5.6}--\eqref{eq:5.8} for a suitable constant $\widetilde \nu_2>0$ independent of $R\geq \widetilde R_2$.

\smallskip

\noindent
{\sl Step 3: Conclusion}

\smallskip

\noindent
Setting $R_1=\max\{ \widetilde R_1,\widetilde R_2\}$, $\nu_1=\min\{\widetilde \nu_1,\widetilde \nu_2\}$, we have the conclusion of 
Proposition \ref{pro:5.2}.
\end{proof}

%%%%%%%%%%%%%%%%%%%%%%%

\noindent
We have the following deformation result from the previous lemma.

\begin{lem}\label{lem:5.3}
Let $0<r_1<r_2\leq r_0$ and $\rho>0$ and suppose that $R$, $P$, $L$ satisfies \eqref{sharp} and
$R\geq R_1$, where $R_1$ is given in Proposition \ref{pro:5.2}.  
Moreover assume that
    \be\label{eq:5.13}
    \K_L\cap B_{r_2}^{(L)}(\wwP) = \emptyset.
    \ee
Then, for any $0<\epsilon<\frac 1 4\nu_1(r_2-r_1)$ and $\overline\epsilon>\epsilon$ there exists 
$\widetilde\eta\in C(E_L,E_L)$ such that
\begin{enumerate}[(i)]
\item $\widetilde \eta(u)=u$ for all $u\in (E_L\setminus B_{r_2}^{(L)}(\wwP))\cup [J_L\leq 2b_\infty-\overline \epsilon]_L$;
\item $J_L(\widetilde \eta(u)) \leq J_L(u)$ for all $u\in E_L$;
\item if $u\in B_{r_1}^{(L)}(\wwP)\cap [J_L\leq 2b_\infty+\epsilon]_L$, then 
$\widetilde\eta(u) \in B_{r_2}^{(L)}(\wwP)\cap [J_L\leq 2b_\infty-\epsilon]_L$;
\item if $u\in B_{r_1}^{(L)}(\wwP)$, then $\widetilde\eta(u)\in B_{r_2}^{(L)}(\wwP)$;
\item for $u\in B_{r_1}^{(L)}(\wwP)$,
if $J_{\A{R}}(u)\leq\rho$, then $J_{\A{R}}(\widetilde\eta(u))\leq\rho$.
\end{enumerate}
\end{lem}

\begin{proof}
It follows from \eqref{eq:5.13} that for some constant $\nu_L>0$
    \be\label{eq:5.14}
    \norm{J_L'(u)}_{E_L^*} \geq \nu_L \quad \text{for all}\ u \in B_{r_2}^{(L)}(\wwP).
    \ee
In fact, if a sequence $(u_j)_{j=1}^\infty\subset B_{r_2}^{(L)}(\wwP)$ satisfies 
$\norm{J_L'(u_j)}_{E_L^*}\to 0$, we can see that, after extracting a subsequence,
$u_j\to u_0\in B_{r_2}^{(L)}(\wwP)$ in $E_L$ and $J_L'(u_0)=0$.  This is a contradiction to \eqref{eq:5.13}.
We may assume $\nu_L<\nu_1$ without loss of generality.

By Proposition \ref{pro:5.2} and \eqref{eq:5.14}, there exists a locally Lipschitz continuous vector field
$\V:\, B_{r_2}^{(L)}(\wwP)\to E_L$ such that
\begin{enumerate}[(1)]
\item $\norm{\V(u)}_{E_L} \leq 1$ for all $u\in E_L$;
\item $J_L'(u)\V(u)\geq \half\nu_L$ for all $u\in B_{r_2}^{(L)}(\wwP)$;
\item $J_L'(u)\V(u)\geq \half\nu_1$ for all $u\in B_{r_2}^{(L)}(\wwP)\setminus B_{r_1}^{(L)}(\wwP)$;
\item $J_{\A{R}}'(u) \V(u)\geq 0$ for all $u\in B_{r_2}^{(L)}(\wwP)$ with $J_{\A{R}}(u) \geq \rho$.
\end{enumerate}
Such a vector field can be constructed in a standard way using a partition of unity (c.f. Appendix
of Rabinowitz \cite{R}).

We define locally Lipschitz functions $\chi_1(u)$, $\chi_2(u):\, E_L\to [0,1]$ by
    \begin{eqnarray*}
    &&\chi_1(u)=\begin{cases}
        0 &\text{for}\ u\not\in B_{r_2}^{(L)}(\wwP), \\
        1 &\text{for}\ u\in B_{\frac{r_1+r_2}2}^{(L)}(\wwP),
    \end{cases} \\
    &&\chi_2(u)=\begin{cases}
        1 &\text{for}\ u\in [J_L\geq 2b_\infty-\epsilon]_L, \\
        0 &\text{for}\ u\in [J_L\leq 2b_\infty-\overline\epsilon]_L.
    \end{cases} 
    \end{eqnarray*}
We consider an initial value problem in $E_L$
    \be\label{eq:5.15}
    \left\{ \begin{aligned}
            &\frac{d\eta}{dt} = -\chi_1(\eta)\chi_2(\eta)\V(\eta),\\
            &\eta(0,u)=u.
            \end{aligned}
    \right.
    \ee
\eqref{eq:5.15} has a unique solution $\eta(t,u)\in C([0,\infty)\times E_L,E_L)$ which satisfies for all 
$t\in [0,\infty)$ and $u\in E_L$
\begin{enumerate}[(1)]
\item $\eta(t,u)=u$ for all $u\in (E_L\setminus B_{r_2}^{(L)}(\wwP))\cup [J_L\leq 2b_\infty-\overline \epsilon]_L$;
\item $\frac d {dt} J_L(\eta(t,u)) \leq 0$ for all $u\in E_L$;
\item $\frac d {dt} J_L(\eta(t,u)) \leq -\frac{\nu_1}2$ if $\eta(t,u)\in (B_{\frac{r_1+r_2}2}^{(L)}(\wwP)
\setminus B_{r_1}^{(L)}(\wwP))\cap [J_L\geq 2b_\infty-\epsilon]_L$;
\item $\frac d {dt} J_L(\eta(t,u)) \leq -\frac{\nu_L}2$ if $\eta(t,u)\in B_{r_1}^{(L)}(\wwP)
\cap [J_L\geq 2b_\infty-\epsilon]_L$;
\item $\frac d {dt} J_{\A{R}}(\eta(t,u)) \leq 0$ if $\eta(t,u)\in B_{r_2}^{(L)}(\wwP)\cap 
[J_{\A{R}}\geq \rho]_L$.
\end{enumerate}
Now we set $\widetilde\eta(u)=\eta(\frac{4\epsilon}{\nu_L},u)$.  Then $\widetilde \eta\in C(E_L,E_L)$
has the desired properties (i)--(v).

Here we show just (iii).  
We argue indirectly and suppose that $u\in B_{r_1}^{(L)}(\wwP)\cap [J_L\leq 2b_\infty+\epsilon]_L$ 
satisfies 
    \be\label{eq:5.16}
    J_L(\eta(t,u)) > 2b_\infty -\epsilon \quad \text{for all}\  t\in [0,\frac {4\epsilon}{\nu_L}].
    \ee
and consider two cases:
\begin{quotation}
\item[Case 1:] $\eta(t,u)\in B_{\frac{r_1+r_2}2}^{(L)}(\wwP)$ for all $t\in [0, \frac{4\epsilon}{\nu_L}]$,
\item[Case 2:] $\eta(t,u)\not\in B_{\frac{r_1+r_2}2}^{(L)}(\wwP)$ for some $t\in [0, \frac{4\epsilon}{\nu_L}]$.
\end{quotation}
If Case 1 occurs, we have
    $$
    \frac d{dt}J_L(\eta(t,u)) = -J_L'(\eta(t,u))\V(\eta(t,u))
    \leq -\frac{\nu_L}2 \quad \text{for all}\ t\in [0, \frac{4\epsilon}{\nu_L}].
    $$
Thus
    $$  J_L(\widetilde\eta(u))= J_L(\eta(\frac{4\epsilon}{\nu_L},u))\leq J_L(u) -2\epsilon
		\leq 2b_\infty-\epsilon.
    $$
which is in contradiction to \eqref{eq:5.16}.

If Case 2 occurs, we can find an interval $[t_0,t_1]\subset [0, \frac{4\epsilon}{\nu_L}]$ such that
    \begin{eqnarray*}
    &&\eta(t_0,u)\in \partial B_{r_1}^{(L)}(\wwP), \ \eta(t_1,u)\in \partial B_{\frac{r_1+r_2}2}^{(L)}(\wwP), \\
    &&\eta(t,u)\in B_{\frac{r_1+r_2}2}^{(L)}(\wwP)\setminus B_{r_1}^{(L)}(\wwP) \ \text{for all}\ t\in (t_0,t_1).
    \end{eqnarray*}
Since $\norm{\frac d{dt}\eta(t,u)}_{E_L} \leq \norm{\V(\eta(t,u))}_{E_L} \leq 1$ for all $t$, 
we have $t_1-t_0\geq \frac{r_2-r_1}2$.  Thus we have
    \begin{eqnarray*}
    J_L(\widetilde \eta(u)) &=& J_L(\eta(\frac{4\epsilon}{\nu_L},u)) \leq J_L(\eta(t_1,u)) \\
    &\leq& J_L(\eta(t_0,u)) - \frac{\nu_1}2 (t_1-t_0) \\
    &\leq& J_L(\eta(t_0,u)) - \frac{\nu_1}2 \frac{r_2-r_1}2 \\
    &\leq& J_L(u) -\frac{\nu_1}2 \frac{r_2-r_1}2,
    \end{eqnarray*}
which is also in contradiction to \eqref{eq:5.16}.
\end{proof}

%%%%%%%%%%%%%%%%%%%%%%%

\subsection{Minimizing problem in $\A{R}$}
To find a critical point of $J_L(u)$, we need to solve the minimizing problem in $\A{R}$.
A decay property of a unique minimizer will play an important role later.

For $R$, $P$, $L$ with $R\geq 1$, \eqref{sharp} and $\rho>0$, we set
    $$  O(\A{R},\rho)=\{ u\in E_L;\, \norm u_{E_L(\A{R})} \leq r_0,\, J_{\A{R}}(u) \leq \rho\}
    $$
and for $u\in O(\A{R},\rho)$ we also set
    $$  K_{\A{R}}(u)=\{ v\in E_L;\, v=u\ \text{in}\ D_L\setminus\A{R}, \
        \norm{v}_{E_L(\A{R})} \leq r_0\}.
    $$
We have the following existence result.

\begin{pro}\label{pro:5.4}
Suppose $\rho\in (0,\frac 1 4 r_0^2)$.  Then for any $u\in O(\A{R},\rho)$, the following
minimizing problem has a unique minimizer $v=v(\A{R};u)\in K_{\A{R}}(u)$:
    \be\label{eq:5.17}
    \inf_{v\in K_{\A{R}}(u)} J_{\A{R}}(v).
    \ee
Moreover
\begin{enumerate}[(i)]
\item $O(\A{R},\rho)\to E_L$;, $u\mapsto v(\A{R};u)$ is continuous;
\item $\norm{v(\A{R};u)}_{E_L(\A{R})} < r_0$ for all $u\in O(\A{R},\rho)$;
\item $v(\A{R};u)(x)=0$ in $\A{R}$ if $u=0$ in $\A{R}$;
\item $J_L(v(\A{R};u))\leq J_L(u)$, $J_{\A{R}}(v(\A{R};u)) \leq J_{\A{R}}(u)$ 
for all $u\in O(\A{R},\rho)$;

\end{enumerate}
\end{pro}

\begin{proof}
By Lemma \ref{lem:5.1} (I), we have
\begin{enumerate}[(1)]
\item $J_{\A{R}}(v)\geq \frac 1 4 \norm v_{E_L(\A{R})}^2$ for all $\norm v_{E_L(\A{R})}\leq r_0$;
\item $J_{\A{R}}(v)$ is strictly convex on $K_{\A{R}}(u)$.
\end{enumerate}
Thus, under the assumption $\rho\in (0,\frac 1 4 r_0^2)$, we have
    $$  \inf_{v\in K_{\A{R}}(u)} J_{\A{R}}(v) \leq J_{\A{R}}(u) \leq \rho < \frac 1 4 r_0^2 \leq
		\inf_{v\in \partial K_{\A{R}}(u)} J_{\A{R}}(v)
    $$ 
and the infimum \eqref{eq:5.17} is achieved in ${\rm int}\, K_{\A{R}}(u)$.
Moreover the minimizer $v(\A{R};u)\in {\rm int}\, K_{\A{R}}(u)$ is unique.  It is easy to see that
(i)--(iv) hold.

\end{proof}

\noindent
We have the following decay estimate for the unique minimizer $v(\A{R};u)$ obtained in 
Proposition \ref{pro:5.4}.

%%%%%%%%%%%%%%%%%%%%%%%%%%%%%%%%%%%%%%%%%%%%%%%%%%%%%%%%%%%%%%%%%%%%%%%%%%%%%%%%%%%%

\begin{lem}\label{lem:5.5}
There exist constants $R_2>0$, $A_1$, $A_2>0$ independent of $P$, $L$ such that
if $R\geq R_2$, then
    $$  \abs{v(\A{R};u)(x)}, \ \abs{\nabla v(\A{R};u)(x)} \leq A_1e^{-A_2R}
        \quad \text{for all}\ x\in \A{2R}.
    $$
\end{lem}

\begin{proof}%[Proof of Lemma \ref{lem:5.5}]

We know  $v(x)=v(\A{R};u)(x)$ is the unique solution in $K_{\A{R}}(u)$ for 
    $$%\be%label{eq:5.18}%
    \left\{
    \begin{aligned}-&\Delta v + V(x)v+Q(x)h(v)-Q(x)f(v)=0 \quad\text{in} \ \A{R},\\
    &v=u \qquad\qquad\qquad\qquad\qquad\qquad\qquad\qquad \text{in}\ D_L\setminus \A{R}.
    \end{aligned}\right. 
    $$
By (VI) of Lemma \ref{lem:5.1}, we see that $v(x)$ satisfies
	$$	\abs{v(x)} \leq \frac 1 e \quad \text{in}\ \A{R+1}.
	$$
In particular, $v(x)$ solves
    $$
    %\left\{
    %\begin{aligned}
	-\Delta v + V(x)v+Q(x)h(v)=0 \quad\text{in} \ \A{R+1}.
	%    &\absv=u \qquad\qquad\qquad\qquad\qquad\quad\text{in}\ D_L\setminus \A{R}.
    %\end{aligned}\right. 
    $$
By the maximal principle, we can get the exponential decay of $v$ in $\A{\frac 3 2 R}$, that is,
	$$	\abs{v(x)} \leq A_1e^{-A_2R} \quad \text{in}\ \A{\frac 3 2 R}.
	$$
By the regularity argument, we also have the exponential decay of $\nabla v$ in $\A{2R}$.
\end{proof}

%%%%%%%%%%%%%%%%%%%%%%%%%%%%%%%%%%%%%%%%%%%%%%%%%%%%%%%%%%%%%%%%%%
%%%%%%%%%%%%%%%%%%%%%%%%%%%%%%%%%%%%%%%%%%%%%%%%%%%%%%%%%%%%%%%%%%

\subsection{Proof of Theorem \ref{thm:2}}
For any given $r\in (0,\half r_0]$ we try to find a critical point $J_L(u)$ in $B_r^{(L)}(\wwP)$
for large $R$, $P$, $L$ with \eqref{sharp} to prove Theorem \ref{thm:2}.

\begin{proof}[Proof of Theorem \ref{thm:2}]
Proof of Theorem \ref{thm:2} is divided into several steps.

\smallskip

\noindent
{\sl Step 1: Choice of parameters $r_1$, $r_2$, $\rho$, $\epsilon$, $\overline\epsilon$}

\smallskip

\noindent
First we choose parameters which will be used in the proof of Theorem \ref{thm:2}.

For any given $r\in (0,\half r_0]$, we set $r_1=\half r$, $r_2=r$, $\rho=\frac 1 8 r^2$.
We apply Proposition \ref{pro:5.2} to obtain $R_1=R_1(r_1,r_2,\rho)$, 
$\nu_1=\nu_1(\rho)>0$.

Next we choose $T>1$ such that
    $$  J_\infty(T\omega), \ J_\infty(T\omega') <0.
    $$
We also choose $\delta_0>0$ with a property: for $\theta>0$
    \begin{align} 
    &J_\infty(\theta \omega) \geq b_\infty-\delta_0 \ (\text{resp.} \ J_\infty(\theta \omega') \geq b_\infty-\delta_0)
    \ \text{implies} \nonumber\\
    &\quad \norm{\theta \omega-\omega}_{E_\infty} < \frac r 4 \ 
    (\text{resp.}\ \norm{\theta \omega'-\omega'}_{E_\infty} < \frac r 4).  \label{eq:5.18}
    \end{align}
Next we set $\overline\epsilon=b_\infty/3$ and we choose
    $$  \epsilon\in \left(0,\min\{\overline\epsilon,\frac {\delta_0} 2,\frac{\nu_1} 4(r_2-r_1)\}\right).
    $$
To show Theorem \ref{thm:2}, we argue indirectly and assume
	$$	\K_L\cap B_r^{(L)}(\wwP)=\emptyset.
	$$
The by Lemma \ref{lem:5.3}, there exists $\widetilde\eta\in C(E_L,E_L)$
satisfying properties (i)--(v) of Lemma \ref{lem:5.3}.

For $(\theta_1,\theta_2)\in [0,T]^2$, we set
    $$  G(\theta_1,\theta_2)(x)=\theta_1\Phi_L({\psi_R\omega})(x) + \theta_2\Phi_L({\psi_R\omega'})(x-P).
    $$
We choose 
    \be\label{eq:5.19}
    R_3\geq \max\{ R_1, R_2\}
    \ee
such that for $R$, $P$, $L$ with $R\geq R_3$ and \eqref{sharp}
    \begin{eqnarray}
    &&J_L(G(\theta_1,\theta_2)) \leq b_\infty+\overline\epsilon \quad \text{for all}\
    (\theta_1,\theta_2)\in \partial([0,T]^2);  \label{eq:5.20}\\
    &&J_L(G(\theta_1,\theta_2)) \leq 2b_\infty+\epsilon \quad \text{for all}\
    (\theta_1,\theta_2)\in [0,T]^2;  \label{eq:5.21}\\
    &&J_L(G(\theta_1,\theta_2)) \geq 2b_\infty-\epsilon \ \text{implies}\ 
        G(\theta_1,\theta_2)\in B_{r/2}^{(L)}(\wwP);  \label{eq:5.22}\\
    &&G(\theta_1,\theta_2)(x) =0 \quad \text{in}\ \A{R}\ \text{for all}\ (\theta_1,\theta_2)\in [0,T]^2;
    \label{eq:5.23}\\
    &&b_L\geq b_\infty -\frac\epsilon 8 \quad \text{for all}\ L\geq 10R; \label{eq:5.24}\\
    && A_1e^{-2A_2R}(3+\norm V_{L^\infty}+2\norm Q_{L^\infty}A_2R)\omega_{N-1}(2R+1)^N <\frac \epsilon 2;
        \label{eq:5.25}
    \end{eqnarray}
where $A_1$, $A_2$ are constant appeared in Lemma \ref{lem:5.5} (c.f. \eqref{eq:5.29}).\\
To find $R_3$, we note that
    \begin{eqnarray}
    &&J_L(\theta_1\Phi_L({\psi_R \omega}))\to J_\infty(\theta_1\omega)\leq b_\infty; \nonumber \\
    &&J_L(\theta_2\Phi_L({\psi_R \omega'}))\to J_\infty(\theta_2\omega')\leq b_\infty; \nonumber \\
    &&J_L(G(\theta_1,\theta_2))=J_L(\theta_1\Phi_L({\psi_R \omega}))+J_L(\theta_2\Phi_L({\psi_R \omega'})) \nonumber\\
        &&\qquad\qquad\qquad\to J_\infty(\theta_1 \omega)+J_\infty(\theta_2 \omega'); \nonumber \\
    &&\norm{(1-\psi_R)\omega}_{E_\infty},\ \norm{(1-\psi_R)\omega}'_{E_\infty}\to 0;
        \label{eq:5.26}
    \end{eqnarray}
as $R\to\infty$ uniformly in $(\theta_1,\theta_2)\in [0,T]^2$.

We can easily find $R_3\geq\max\{ R_1,R_2\}$ such that \eqref{eq:5.20}--\eqref{eq:5.21} hold for $R\geq R_3$.  
For \eqref{eq:5.22},
we note that $J_L(G(\theta_1,\theta_2))\geq 2b_\infty-\epsilon$ implies $J_L(\theta_1\Phi_L({\psi_R\omega}))$,
$J_L(\theta_2\Phi_L({\psi_R\omega'}))\geq b_\infty-\frac 3 2\epsilon$ for large $R$ and thus 
$J_\infty(\theta_1\omega)$, $J_\infty(\theta_2\omega)\geq b_\infty-2\epsilon\geq b_\infty-\delta_0$ for large $R$.  
Thus from \eqref{eq:5.18} and \eqref{eq:5.26}, we have
\eqref{eq:5.22}.  Properties \eqref{eq:5.23} and \eqref{eq:5.24} follow from the definition of $G(\theta_1,\theta_2)$ and 
Theorem \ref{thm:3}.  Property \eqref{eq:5.25} is easily checked.

\smallskip

\noindent
{\sl Step 2: Definition of $\widehat G(\theta_1,\theta_2)$ and its properties}

\smallskip

\noindent
We set
    $$  \widehat G(\theta_1,\theta_2)=\widetilde\eta(G(\theta_1,\theta_2))\in C([0,T]^2,E_L),
    $$
where $\widetilde\eta\in C(E_L,E_L)$ is defined in Lemma \ref{lem:5.3}.
$\widehat G(\theta_1,\theta_2)$ has the following properties:
\begin{enumerate}[(1)]
\item $\widehat G(\theta_1,\theta_2)=G(\theta_1,\theta_2)$ for all $(\theta_1,\theta_2)\in \partial([0,T]^2)$;
\item $J_L(\widehat G(\theta_1,\theta_2))\leq 2b_\infty-\epsilon$ for all $(\theta_1,\theta_2)\in [0,T]^2$;
\item $\norm{\widehat G(\theta_1,\theta_2)}_{E_L(\A{R})}\leq r_0$,\ 
$J_{\A{R}}(\widehat G(\theta_1,\theta_2)) \leq \frac 1 8r^2$ for all $(\theta_1,\theta_2)\in [0,T]^2$.  \\
In particular, $\widehat G(\theta_1,\theta_2)\in O(\A{R},\frac 1 8 r^2)$ for all $(\theta_1,\theta_2)\in [0,T]^2$.
\end{enumerate}
(1) and (2) follow from Lemma \ref{lem:5.3} and \eqref{eq:5.20}--\eqref{eq:5.22}.  For (3), we note that
$\widetilde\eta(u)=u$ or $\widetilde\eta(u)\in B_{r_2}^{(L)}(\wwP)$ holds for all $u\in E_L$.
By \eqref{eq:5.23}, we have $\norm{\widehat G(\theta_1,\theta_2)}_{E_L(\A{R})} \leq r_0$ for all
$(\theta_1,\theta_2)\in [0,T]^2$.  We also deduce \\
$J_{\A{R}}(\widehat G(\theta_1,\theta_2)) \leq\frac 1 8r^2$ from Lemma \ref{lem:5.3}.

\smallskip

\noindent
{\sl Step 3: Definition of $\hhG(\theta_1,\theta_2)$ and its properties}

\smallskip

\noindent
By the property (3) of $\widehat G(\theta_1,\theta_2)$, we can define
    $$  \hhG(\theta_1,\theta_2)=v(\A{R};\widehat G(\theta_1,\theta_2))\in C([0,T]^2,E_L).
    $$
$\hhG(\theta_1,\theta_2)$ has the following properties:
\begin{enumerate}[(1)]
\item \be\label{eq:5.27}
    \hhG(\theta_1,\theta_2)=G(\theta_1,\theta_2)\ \text{for all}\  (\theta_1,\theta_2)\in \partial([0,T]^2);
    \ee
\item $J_L(\hhG(\theta_1,\theta_2))\leq 2b_\infty-\epsilon$ for all $(\theta_1,\theta_2)\in [0,T]^2$;
\item $\norm{\hhG(\theta_1,\theta_2)}_{E_L(\A{R})}\leq r_0$
for all $(\theta_1,\theta_2)\in [0,T]^2$;
\item $\abs{\hhG(\theta_1,\theta_2)}$, $\abs{\nabla\hhG(\theta_1,\theta_2)}\leq A_1e^{-A_2R}$
for all $x\in \A{2R}$ and $(\theta_1,\theta_2)\in [0,T]^2$.
\end{enumerate}
These properties follow from Proposition \ref{pro:5.4} and Lemma \ref{lem:5.5}.

\smallskip

\noindent
{\sl Step 4: Definition of $\oG(\theta_1,\theta_2)$ and its properties}

\smallskip

\noindent
We choose $2L$-periodic functions $\zeta_1(x)$, $\zeta_2(x)\in C^\infty(\R^N,\R)$ such that
    \begin{eqnarray*}
    &&\zeta_1(x)=\begin{cases}
        1   &\text{for}\ x\in B_{2R}(0),\\
        0   &\text{for}\ x\in D_L\setminus B_{2R+1}(0),
        \end{cases}\\
    &&\zeta_2(x)=\begin{cases}
        1   &\text{for}\ x\in B_{2R}(P),\\
        0   &\text{for}\ x\in D_L\setminus B_{2R+1}(P),
        \end{cases}\\
    &&\zeta_1(x),\, \zeta_2(x)\in [0,1] \ \text{for all}\ x\in \R^N,\\
    &&\abs{\nabla\zeta_1(x)},\, \abs{\nabla\zeta_2(x)} \leq 2 \ \text{for all}\ x\in \R^N.
    \end{eqnarray*}
We define $\overline g_1(\theta_1,\theta_2)$, $\overline g_2(\theta_1,\theta_2)$, 
$\oG(\theta_1,\theta_2):\, [0,T]^2\to E_L$ by
    \begin{eqnarray*}
    &&\overline g_1(\theta_1,\theta_2)(x)=\zeta_1(x)\hhG(\theta_1,\theta_2)(x), \\
    &&\overline g_2(\theta_1,\theta_2)(x)=\zeta_2(x)\hhG(\theta_1,\theta_2)(x), \\
    &&\oG(\theta_1,\theta_2)(x)=\zeta_1(x)\hhG(\theta_1,\theta_2)(x)+\zeta_2(x)\hhG(\theta_1,\theta_2)(x).
    \end{eqnarray*}
$\overline g_1(\theta_1,\theta_2)$, $\overline g_2(\theta_1,\theta_2)$, $\oG(\theta_1,\theta_2)$ have
the following properties:
\begin{enumerate}[(1)]
\item for $(\theta_1,\theta_2)\in\partial([0,T]^2)$
    \begin{eqnarray}
    &&\oG(\theta_1,\theta_2)=G(\theta_1,\theta_2),    \nonumber\\
    &&\overline g_1(\theta_1,\theta_2)(x)=\theta_1\Phi_L({\psi_R \omega})(x), \nonumber\\
    &&\overline g_2(\theta_1,\theta_2)(x)=\theta_2\Phi_L({\psi_R \omega'})(x-P) \label{eq:5.28}
    \end{eqnarray}
\item $J_L(\oG(\theta_1,\theta_2)) \leq 2b_\infty-\half\epsilon$ for all 
$(\theta_1,\theta_2)\in [0,T]^2$.
\end{enumerate}
(1) follow from \eqref{eq:5.27}.
To see (2), we write $\hhu=\hhG(\theta_1,\theta_2)$, $\ou=\oG(\theta_1,\theta_2)$ and we compute
    $$  J_L(\ou)-J_L(\hhu) = J_{\A{2R}}(\ou)-J_{\A{2R}}(\hhu) \leq J_{\A{2R}}(\ou).
    $$
Here we used Lemma \ref{lem:5.1} (III).
Since $\abs{\ou}\leq \abs{\hhu}\leq A_1e^{-A_2R}$, $\abs{\nabla\ou}\leq 2\abs{\hhu} +\abs{\nabla\hhu}
\leq 3A_1e^{-A_2R}$ in $\A{2R}$, we have
    \begin{eqnarray}
    && J_{\A{2R}}(\ou) \nonumber\\
    &=&\half \int_{\A{2R}\setminus \A{2R+1}} \abs{\nabla\ou}^2 +V(x)\ou^2
        -Q(x)\ou^2\log\ou^2 \nonumber\\
    &\leq& \half A_1^2 e^{-2A_2R}\left(3+\norm V_{L^\infty} +\norm Q_{L^\infty}(-\log A_1^2+2A_2R)\right)\nonumber\\
    &&\qquad\qquad  \times \text{meas}(\A{2R}\setminus \A{2R+1})\nonumber\\
    &\leq& A_1^2e^{-2A_2R}(3+\norm V_{L^\infty}+2\norm Q_{L^\infty} A_2R)\omega_{N-1}(2R+1)^N. \label{eq:5.29}
    \end{eqnarray}
By our choice of $R_3$, we have $J_{\A{2R-1}}(\ou)<\half\epsilon$.  Thus we have
    \begin{eqnarray*}
    J_L(\ou) &=& J_L(\hhu) +(J_L(\ou) - J_L(\hhu)) \leq  J_L(\hhu) + J_{\A{2R}}(\ou) \\
    &\leq& 2b_\infty -\epsilon + \half\epsilon = 2b_\infty -\half \epsilon.
    \end{eqnarray*}
Thus we get (2).

\smallskip

\noindent
{\sl Step 5: An intersection result and the end of proof}

\smallskip

\noindent
By the property \eqref{eq:5.28}, for any curve $\gamma(s):\, [0,1]\to [0,T]^2$ with $\gamma(0)\in \{0\}\times [0,T]$, 
$\gamma(1)\in \{T\}\times [0,T]$ (resp. $\gamma(0)\in [0,T]\times \{0\}$, 
$\gamma(1)\in [0,T]\times \{T\}$), a path $\overline g_1(\gamma(s))$ (resp. $\overline g_2(\gamma(s))$)
is a path joining $0$ and $T\Phi_L({\psi_R\omega})$ (resp. $T\Phi_L({\psi_R\omega'})(\cdot -P)$).
Noting $J_L(T\Phi_L({\psi_R\omega}))$, $J_L(T\Phi_L({\psi_R\omega'}))<0$, they can be regarded as a sample path
corresponding to mountain pass theorem. 

As in Proposition 3.4 of Coti Zelati-Rabinowitz \cite{CZR3}, there exists a $(\overline\theta_1,\overline\theta_2)
\in [0,T]^2$ such that
    $$  J_L(\overline g_1(\overline\theta_1,\overline\theta_2))\geq b_L,\quad
        J_L(\overline g_2(\overline\theta_1,\overline\theta_2))\geq b_L.
    $$
Thus we have
    \begin{eqnarray*}
    J_L(\oG(\overline\theta_1,\overline\theta_2)) &=& J_L(\overline g_1(\overline\theta_1,\overline\theta_2))
        + J_L(\overline g_2(\overline\theta_1,\overline\theta_2))\\
    &\geq& 2b_L.
    \end{eqnarray*}
Therefore we have
    $$  2b_\infty -\half\epsilon \geq 2b_L.
    $$
This contradicts with \eqref{eq:5.24} and thus $\K_L\cap B_r^{(L)}(\wwP)\not=\emptyset$.
Thus, choosing $R_{r0}\geq R_3$, we complete the proof of Theorem \ref{thm:2}.
\end{proof}

%%%%%%%%%%%%%

\subsection{Proof of Theorem \ref{thm:1}}
Finally we give a proof of our Theorem \ref{thm:1}.

\begin{proof}[Proof of Theorem \ref{thm:1}]
Let $R_3>0$ be a number given in \eqref{eq:5.19}.  We take $R_{r1}\geq R_3$ such that
    \be\label{eq:5.30}
    \norm{\omega-\psi_{R_{r1}}\omega}_{E_\infty},\ \norm{\omega'-\psi_{R_{r1}}\omega'}_{E_\infty} \leq \frac r 2.
    \ee
We fix $P\in \Z^N$ with $\abs P\geq 5{R_{r1}}$.  By Theorem \ref{thm:2}, for any $L\geq 2\abs P$ there
exists a critical point $u_L\in \K_L\cap B_r^{(L)}(\Omega_{R_{r1},P,L})$.  By Lemma \ref{lem:2.7} (ii) and Lemma \ref{lem:5.1}
we have
    $$  J_L(u_L)=\half\int_{D_L}Q(x)u_L^2\in [2b_\infty-\alpha,2b_\infty+\alpha].
    $$
We apply our concentration-compactness result (Proposition \ref{pro:3.1}) to $u_L$ ($L=2\abs P+1, 2\abs P+2, \cdots$).
After extracting a subsequence $L_j\to\infty$, we have for some $w_0\in \K_\infty$
    \begin{eqnarray}
    &&\norm{u_{L_j}-\Phi_{L_j}(\psi_{L_j/2}w_0)}_{E_{L_j}} \to 0,  \label{eq:5.31}\\
    &&J_{L_j}(u_{L_j}) \to J_\infty(w_0) \quad \text{as}\ j\to\infty.  \label{eq:5.32}
    \end{eqnarray}
In fact, if not, we have $m\geq 2$ in the statement of Proposition \ref{pro:3.1} and for some
sequence $(y_j)_{j=1}^\infty\subset\R^N$ with $y_j\in D_{L_j}$ and $\abs{y_j}\to\infty$
    $$  \liminf_{j\to\infty} \norm{u_{L_j}}_{H^1(B_{L_j/2}(y_j))}>0,
    $$
which is in a contradiction to $u_L\in B_r^{(L)}(\Omega_{R_{r1},P,L})$.  
It easily follows from \eqref{eq:5.31}--\eqref{eq:5.32} that
    \begin{eqnarray*}
    &&\norm{w_0-\Omega_{R_{r1},P,L}}_{E_\infty} \\
    &=& \lim_{j\to\infty}\norm{\Phi_{L_j}(\psi_{L_j/2}w_0)-\Omega_{R_{r1},P,L}}_{E_{L_j}} \\
    &\leq& \limsup_{j\to\infty} \norm{\Phi_{L_j}(\psi_{L_j/2}w_0)-u_{L_j}}_{E_{L_j}} 
        +  \limsup_{j\to\infty} \norm{u_{L_j} -\Omega_{R_{r1},P,L}}_{E_{L_j}} \\
    &\leq& r,
    \end{eqnarray*}
which implies by \eqref{eq:5.30}
    $$  \norm{w_0-(\omega+\omega'(\cdot-P))}_{E_\infty} \leq 2r.
    $$
\end{proof}
%%%%%%%%%%%%%%%%%%%%%%%%%%%%%%%%
%%%%%%%%%%%%%%%%%%%%%%%%%%%%%%%%

\appendix

\setcounter{equation}{0}
\section{Proofs of Proposition \ref{pro:3.1} and (VI) of Lemma \ref{lem:5.1}}
\subsection{Proof of Proposition \ref{pro:3.1}}
In the following proof, we use an idea from Jeanjean-Tanaka \cite{JT}.

\begin{proof}[Proof of Proposition \ref{pro:3.1}]
We assume $L_j$, $u_j$ ($j=1,2,\cdots$) satisfy the assumption \eqref{eq:3.1} of Proposition \ref{pro:3.1}.
By Lemma \ref{lem:2.7}, we have for some $A>0$ independent of $j$
    $$  \norm{u_j}_{E_{L_j}}^2, \ \int_{D_{L_j}} H(u_j), \ \int_{D_{L_j}} h(u_j)u_j \leq A
        \quad \text{for all}\ j.
    $$

\smallskip

\noindent
{\sl Step 1: After extracting a subsequence, there exists a sequence $y_j^1\in \Z^N$ and 
$w^1\in\K_\infty \setminus\{ 0\}$ such that
    \begin{eqnarray*}
    &&y_j^1\in D_{L_j}, \nonumber\\
    &&u_j(x+y_j^1)\wlimit w^1(x) \quad \text{weakly in}\ H^1_{loc}(\R^N).  
    \end{eqnarray*}
}

\smallskip

\noindent
For $q\in(2,2^*)$, we set
    $$  d_j = \sup_{n\in \Z^N}\norm{u_j}_{L^q(D_1(n))} \quad \text{for}\ j=1,2,\cdots.
    $$
If $d_j\to 0$ as $j\to\infty$, by Lemma \ref{lem:3.3} we have $J_{L_j}(u_j)\to 0$,
which contradicts with \eqref{eq:3.1}.  Thus, after extracting a subsequence if necessary, we may assume
$d_j\to d_0>0$ and there exists $y_j^1\in \Z^N$ such that
    $$  \norm{u_j(\cdot+y_j^1)}_{L^1(D_1(0))} \to d_0>0.
    $$
We may also assume that there exists $w^1\in H^1_{loc}(\R^N)\setminus\{ 0\}$ such that
    $$  u_j(\cdot+y_j^1) \wlimit w^1 \quad \text{weakly in}\  H^1_{loc}(\R^N).
    $$
We claim $w^1\in \K_\infty\setminus\{ 0\}$.  In fact, for any $L\in \N$
    \begin{eqnarray*}
    \norm{w^1}_{E_\infty(D_L)}^2 &\leq& \limsup_{j\to\infty} \norm{u_j(\cdot+y_j^1)}_{E_{L_j}(D_L)}^2 \\
    &\leq& \limsup_{j\to\infty}\norm{u_j}_{E_{L_j}}^2\\
    &\leq& A, \\
    \int_{D_L} H(w^1) &\leq& \limsup_{j\to\infty} \int_{D_L} H(u_j(\cdot+y_j^1)) \\
    &\leq& \limsup_{j\to\infty} \int_{D_{L_j}} H(u_j(\cdot+y_j^1)) \\
    &\leq& A.
    \end{eqnarray*}
Since $A$ is independent of $L$, we have $w^1\in E_\infty$ and $\int_{\R^N} H(w^1)<\infty$.  Thus
$w^1\in \D$.
Next we see $w^1\in\K_\infty$.  For any $\varphi\in {C}_0^\infty(\R^N)$ we note that 
$\supp \varphi\subset D_{L_j}$ for large $j$.  Thus
    \begin{eqnarray*}
    &&\int_{\R^N}\nabla w^1\nabla\varphi+V(x)w^1\varphi-Q(x)g(w^1)\varphi\\
    &=& \lim_{j\to\infty}\int_{D_{L_j}} \nabla u_j(\cdot+y_j^1)\nabla\varphi+V(x)u_j(\cdot+y_j^1)\varphi
        -Q(x)g(u_j(\cdot+y_j^1))\varphi \\
    &=& \lim_{j\to\infty} J_{L_j}'(u_j)\varphi(\cdot-y_j^1)=0.
    \end{eqnarray*}
Therefore $w^1\in\K_\infty$.

\smallskip

\noindent
Next we assume that there exists $m_0\in\N$,  $w^\ell\in \K_\infty\setminus\{ 0\}$, $(y_j^\ell)_{j=1}^\infty
\subset\Z^N$ with $y_j^\ell\in D_{L_j}$ ($\ell=1,2,\cdots,m_0$) such that
    \begin{eqnarray}
    &&\dist_{L_j}(y_j^\ell,y_j^{\ell'})\to \infty \quad \text{for}\ \ell\not=\ell', \label{eq:A.1}\\
    &&u_j(\cdot+y_j^\ell)\wlimit w^\ell \ \text{weakly in}\ H^1_{loc}(\R^N) \ 
        \text{for all}\ \ell=1,2,\cdots,m_0. \label{eq:A.2}
    \end{eqnarray}
For $R_j\in\N$ with $R_j\to\infty$ and
$R_{j} \leq L_{j}$, we set
    $$  \widetilde w_j=\sum_{\ell=1}^{m_0} \Phi_{L_j}(\psi_{R_j}w^\ell)(\cdot-y_j^\ell)
        \in E_{L_j}.
    $$
Taking a subsequence if necessary, we may assume
    $$  \lim_{j\to\infty} \norm{u_j-\widetilde w_j}_{E_{L_j}}, \
        \lim_{j\to\infty} \norm{u_j}_{E_{L_j}}  \quad \text{exist}.
    $$
We show

\smallskip

\noindent
{\sl Step 2: $\lim_{j\to\infty} \norm{u_j-\widetilde w_j}_{E_{L_j}}^2 = \lim_{j\to\infty} \norm{u_j}_{E_{L_j}}^2
-\sum_{\ell=1}^{m_0} \norm{w^\ell}_{E_\infty}^2$.
}

\smallskip

\noindent
It follows from \eqref{eq:A.2} and
    $$  \psi_{R_j}w^\ell\to w^\ell\ \text{in}\ E_\infty\ \text{as}\ j\to\infty\ \text{for all}\ \ell=1,2,\cdots,m_0 
    $$
that
    \begin{eqnarray}
    &&\inp{u_j,\, \Phi_{L_j}(\psi_{R_j}w^\ell)(\cdot-y_j^\ell)}_{E_{L_j}}
    = \inp{u_j(\cdot+y_j^\ell),\, \Phi_{L_j}(\psi_{R_j}w^\ell)}_{E_{L_j}} \nonumber\\
    &&\qquad \to \norm{w^\ell}_{E_\infty}^2 \quad \text{for all}\ \ell=1,2,\cdots,m_0, \label{eq:A.3}\\
    && \inp{\Phi_{L_j}(\psi_{R_j}w^\ell)(\cdot-y_j^\ell),\, 
            \Phi_{L_j}(\psi_{R_j}w^{\ell'})(\cdot-y_j^{\ell'})}_{E_{L_j}} \nonumber\\
    &&\qquad \to
        \begin{cases}
        \norm{w^\ell}_{E_\infty}^2 &\text{if}\ \ell=\ell',\\
        0                           &\text{if}\ \ell\not=\ell'
        \end{cases}\label{eq:A.4}
    \end{eqnarray}
as $j\to\infty$. \\
Thus we have 
    \begin{eqnarray*}
    \norm{u_j-\widetilde w_j}_{E_{L_j}}^2 &=& \norm{u_j}_{E_{L_j}}^2 -2 \inp{u_j,\,\widetilde w_j}_{E_{L_j}}
        +\norm{\widetilde w_j}_{E_{L_j}}^2 \\
    &=& \norm{u_j}_{E_{L_j}}^2 -2 \sum_{\ell=1}^{m_0} \inp{u_j, \Phi_{L_j}(\psi_{R_j}w^\ell)(\cdot-y_j^\ell)}_{E_{L_j}} \\
    && \quad + \sum_{\ell=1}^{m_0}\sum_{\ell'=1}^{m_0} \inp{\Phi_{L_j}(\psi_{R_j}w^\ell)(\cdot-y_j^\ell),\,
            \Phi_{L_j}(\psi_{R_j}w^{\ell'})(\cdot-y_j^{\ell'})}_{E_{L_j}} \\
    &\to& \norm{u_j}_{E_{L_j}}^2 - \sum_{\ell=1}^{m_0}\norm{w^\ell}_{E_\infty}^2
        \quad \text{as}\ j\to\infty.
    \end{eqnarray*}

\smallskip

\noindent
Next we set
    $$  \widetilde d_j = \sup_{n\in\Z^N} \norm{u_j-\widetilde w_j}_{L^q(D_1(n))}.
    $$
After extracting a subsequence, we may assume $\lim_{j\to\infty}\widetilde d_j$ exists.  We consider 2 cases:
\begin{quotation}
\item[Case 1:] $\widetilde d_j\to 0$ as $j\to\infty$,
\item[Case 2:] $\widetilde d_j\not\to 0$ as $j\to\infty$.
\end{quotation}

\smallskip

\noindent
{\sl Step 3: If Case 1 occurs, we have $\norm{u_j-\widetilde w_j}_{E_{L_j}}\to 0$ as $j\to\infty$.
}

\smallskip

\noindent
In fact, if $\widetilde d_j\to 0$, we have by Lemma \ref{lem:3.2}
    \be\label{eq:A.5}
    \norm{u_j-\widetilde w_j}_{L^q(D_{L_j})} \to 0.
    \ee
We have
    \begin{eqnarray}
    \norm{u_j-\widetilde w_j}_{E_{L_j}}^2 
    &=& \inp{u_j,\, u_j-\widetilde w_j}_{E_{L_j}} - \inp{\widetilde w_j,\, u_j-\widetilde w_j}_{E_{L_j}} \nonumber\\
    &=& J_{L_j}'(u_j)(u_j-\widetilde w_j) -\int_{D_{L_j}} Q(x)h(u_j)(u_j-\widetilde w_j) \nonumber\\
    && + \int_{D_{L_j}} Q(x)f(u_j)(u_j-\widetilde w_j) - \inp{\widetilde w_j,\, u_j-\widetilde w_j}_{E_{L_j}} 
        \nonumber\\
    &=& o(1) -(I) +(II) -(III) \quad \text{as}\ j\to\infty.
        \label{eq:A.6}
    \end{eqnarray}
By \eqref{eq:A.5}, we can see that
    \be\label{eq:A.7}
    (II)\to 0 \quad \text{as}\ j\to\infty.
    \ee
For (III), we have
    \begin{eqnarray}
    (III)&=& \inp{\sum_{\ell=1}^{m_0}\Phi_{L_j}(\psi_{R_j}w^\ell)(\cdot-y_j^\ell),\, 
                u_j-\sum_{\ell'=1}^{m_0}\Phi_{L_j}(\psi_{R_j}w^{\ell'})(\cdot-y_j^{\ell'})}_{E_{L_j}} \nonumber\\
    &=& \sum_{\ell=1}^{m_0}\inp{\Phi_{L_j}(\psi_{R_j}w^\ell),\, 
            u_j(\cdot+y_j^\ell)-\Phi_{L_j}(\psi_{R_j}w^{\ell})}_{E_{L_j}} \nonumber\\
    &&-\sum_{\ell\not=\ell'} \inp{\Phi_{L_j}(\psi_{R_j}w^\ell),\, 
            \Phi_{L_j}(\psi_{R_j}w^{\ell'})(\cdot+y_j^\ell-y_j^{\ell'})}_{E_{L_j}} \nonumber\\
     &\to& 0 \quad \text{as}\ j\to\infty,  \label{eq:A.8}
     \end{eqnarray}
which follows from \eqref{eq:A.3}--\eqref{eq:A.4}.\\
For (I), we fix $\theta\in (0,1]$ small and $\rho\geq 1$ large.  We use notation:
    $$  B(\rho,\ell,j)=\{ x\in D_{L_j};\, \dist_{L_j}(x,y_j^\ell)<\rho\},
    $$
that is,
    $$  D_{L_j}\setminus B(\rho,\ell,j)= D_{L_j}\setminus\bigcup_{n\in \Z^N} B_\rho(y_j^\ell+2L_jn).
    $$
We note that
    $$  B(\rho,\ell,j)\cap B(\rho,\ell',j) =\emptyset \quad \text{for} \ \ell\not=\ell',
    $$
provided
    \be\label{eq:A.9}
    \rho < \half\min_{\ell\not=\ell'}\dist_{L_j}(y_j^\ell,y_j^{\ell'}).
    \ee
By \eqref{eq:A.1} we remark that for any $\rho\geq 1$ \eqref{eq:A.9} holds for large $j$.

We compute
    \begin{eqnarray*}
    (I) &=& \int_{D_{L_j}} Q(x)h(u_j)(u_j-\widetilde w_j) \\
    &=& \int_{D_{L_j}\setminus\bigcup_{\ell=1}^{m_0}B(\rho,\ell,j) } Q(x)h(u_j)(u_j-\widetilde w_j) 
        + \sum_{\ell=1}^{m_0} \int_{B(\rho,\ell,j)}Q(x) h(u_j) (u_j-\widetilde w_j) \\
    &=& (I1) + \sum_{\ell=1}^{m_0} (I2)_\ell.
    \end{eqnarray*}
By Lemma \ref{lem:2.1} (iii-b),
    \begin{eqnarray*}
    (I1)  
    &\geq& -\int_{D_{L_j}\setminus\bigcup_{\ell=1}^{m_0}B(\rho,\ell,j) } Q(x)h(u_j)\widetilde w_j\\
    &=& - \sum_{\ell'=1}^{m_0} \int_{D_{L_j}\setminus\bigcup_{\ell=1}^{m_0}B(\rho,\ell,j) }
        Q(x)h(u_j)\Phi_{L_j}(\psi_{R_j}w^{\ell'})(\cdot-y_j^{\ell'}) \\
    &\geq& -\norm Q_{L^\infty} \sum_{\ell'=1}^{m_0} \left(
        \theta\int_{D_{L_j}\setminus\bigcup_{\ell=1}^{m_0}B(\rho,\ell,j) } H(u_j) \right.\\
    && \qquad 
        +\left. \frac 1 \theta \int_{D_{L_j}\setminus\bigcup_{\ell=1}^{m_0}B(\rho,\ell,j) } 
            H(\Phi_{L_j}(\psi_{R_j}w^{\ell'})(\cdot-y_j^{\ell'})) \right)\\
    &\geq& -\norm Q_{L^\infty} \sum_{\ell'=1}^{m_0} \left(
        \theta A +\frac 1\theta \int_{\R^N\setminus B_\rho(0)} H(w^{\ell'})\right).
    \end{eqnarray*}
We also have
    \begin{eqnarray*}
    (I2)_\ell
    &=& \int_{B(\rho,\ell,j)} Q(x)h(u_j)(u_j-\sum_{\ell'=1}^{m_0} \Phi_{L_j}(\psi_{R_j}w^{\ell'})(\cdot-y_j^{\ell'})) \\
    &=& \int_{B_\rho(0)} Q(x) h(u_j(\cdot+y_j^{\ell}))(u_j(\cdot+y_j^{\ell})-w^{\ell}) +o(1)\\
    &\to& 0 \quad \text{as}\ j\to\infty.
    \end{eqnarray*}
Thus we have
    $$  \liminf_{j\to\infty}(I) \geq -\norm Q_{L^\infty} \sum_{\ell=1}^{m_0}\left(\theta A
        +\frac 1 \theta \int_{\R^N\setminus B_\rho(0)} H(w^\ell)\right).
    $$
Since $\theta\in (0,1]$ and $\rho\geq 1$ are arbitrary, we have
    \be\label{eq:A.10}
    \liminf_{j\to\infty} (I) \geq 0.
    \ee
Thus by \eqref{eq:A.6}--\eqref{eq:A.10}, we have $\norm{u_j-\widetilde w_j}\to 0$.

\smallskip

\noindent
Next we deal with Case 2.

\smallskip

\noindent
{\sl Step 4: If Case 2 occurs, there exists $w^{m_0+1}\in \K_\infty\setminus\{ 0\}$ and a sequence 
$(y_j^{m_0+1})_{j=1}^\infty\subset\Z^N$ such that
    \begin{eqnarray}
    &&\dist_{L_j}(y_j^\ell,y_j^{m_0+1})\to\infty\quad \text{for all}\ \ell=1,2,\cdots,m_0, \label{eq:A.11}\\
    &&u_j(\cdot+y_j^{m_0+1})\wlimit w^{m_0+1} \quad \text{weakly in}\ H_{loc}^1(\R^N)  \nonumber
    \end{eqnarray}
as $j\to\infty$.
}

\smallskip

\noindent
In fact, we choose $y_j^{m_0+1}\in \Z^N$ such that
    $$  \norm{u_j-\widetilde w_j}_{L^q(D_1(y_j^{m_0+1}))} = \widetilde d_j.
    $$
\eqref{eq:A.11} follows from \eqref{eq:A.1} and \eqref{eq:A.2}.  By \eqref{eq:A.11}, we can see 
	$$	\norm{\Phi_{L_j}(\psi_{R_j}w^\ell)}_{L^q(D_1(y_j^{m_0+1}))}\to 0 \quad\text{as}\  j\to \infty
	$$
for $\ell=1,2,\cdots, m_0$ and thus $u_j(\cdot+y_j^{m_0+1})$ has a non-zero weak limit $w^{m_0+1}$.  
As in Step 1, we can see $w^{m_0+1}\in \K_\infty\setminus\{ 0\}$.

\smallskip

\noindent
{\sl Step 5: Conclusion}

\smallskip

\noindent
We follow a recursive procedure. We start with $m=1$ and use Step 1 to find $w^1$ and $(y_j^1)_{j=1}^\infty$.
If it satisfies 
$\sup_{n\in\Z^N}\norm{u_j-\Phi_{L_j}(\psi_{R_j}w^1)(\cdot-y_j^1)}_{L^q(D_1(n))}\to 0$, we are
done by Step 3.  Otherwise, we use Step 4 to get $w^2$ and $(y_j^2)_{j=1}^\infty$ and 
continue this procedure. Next we prove this procedure stops in finite steps.

By Step 2, we have
    $$  \sum_{\ell=1}^{m} \norm{w^\ell}_{E_\infty}^2 \leq \lim_{j\to\infty}\norm{u_j}_{E_{L_j}}^2
        \leq A.
    $$
On the other hand, by Remark \ref{remark:2.2}
    $$  m\rho_\infty^2 \leq A.
    $$
Thus, this procedure must end in finite steps.  Therefore there exists $(w^\ell)_{\ell=1}^m\subset
\K_\infty\setminus\{ 0\}$ and $(y_j^\ell)_{j=1}^\infty$ such that \eqref{eq:3.2}--\eqref{eq:3.3} hold.
We show here \eqref{eq:3.4} and \eqref{eq:3.5}.

By \eqref{eq:2.8},
    \begin{eqnarray*}
    J_{L_j}(u_j) &=& J_{L_j}(u_j) -\half J_{L_j}'(u_j)u_j +o(1) \\
    &=& \half\int_{D_{L_j}} Q(x)u_j^2 +o(1).
    \end{eqnarray*}
Thus \eqref{eq:3.3} implies
    \be\label{eq:A.12}
    J_{L_j}(u_j) \to \half\sum_{\ell=1}^m \int_{\R^N} Q(x)(w^\ell)^2
    =\sum_{\ell=1}^m J_\infty(w^\ell).
    \ee
This is nothing but \eqref{eq:3.4}.

By \eqref{eq:A.12} and \eqref{eq:3.3}, we have
    \begin{eqnarray*}
    \int_{D_{L_j}} Q(x)H(u_j) &=& J_{L_j}(u_j)-\half\norm{u_j}_{E_{L_j}}^2 +\int_{D_{L_j}} Q(x)F(u_j)\\
    &\to& \sum_{\ell=1}^m \left(J_\infty(w^\ell) +\half\norm{w^\ell}_{E_\infty}^2
        -\int_{\R^N} Q(x)F(w^\ell)\right)\\
    &=& \sum_{\ell=1}^m \int_{\R^N} Q(x)H(w^\ell).
    \end{eqnarray*}
Clearly we have for all $R\geq 1$
    $$  \int_{\bigcup_{\ell=1}^m B(R,\ell,j)} Q(x)H(u_j) 
        \to  \sum_{\ell=1}^m \int_{B_R(0)} Q(x)H(w^\ell).
    $$
Thus for any $\epsilon>0$ we can find a large $R_\epsilon>1$ such that
    $$  \int_{D_{L_j}\setminus \bigcup_{\ell=1}^m B(R_\epsilon,\ell,j)} Q(x)H(u_j)
        <\epsilon,
    $$
from which we can show \eqref{eq:3.5}.
\end{proof}

%%%%%%%%%%%%%%%%%%%%%%%%%%%%%

\subsection{Proof of (VI) of Lemma \ref{lem:5.1}}
We need the following lemma to prove (VI) of Lemma \ref{lem:5.1}.

\begin{lem}[Subsolution estimate, Theorem C.1.2 of \cite{S}]\label{lem:A.1}
Suppose \\
$v\in H^1(B_2(x_0))$ solves
    $$
    -\Delta v+\widehat V(x)v=0 \quad \text{in}\ B_2(x_0).
    $$
Then
    $$  \abs{v(x_0)} \leq C\int_{B_1(x_0)} \abs v,
    $$
where $C>0$ is a constant depending only on the following quantities:
   \begin{eqnarray*}
    &&\sup_{x\in B_1(x_0)}\int_{\abs{y-x}\leq1} {\widehat V(y)}^-\, dy\quad \mbox{if}\ N=1;\\
    &&\sup_{x\in B_1(x_0)}\int_{\abs{y-x}\leq\half}\log(\abs{x-y}^{-1})\,{\widehat V(y)}^-\, dy\quad \mbox{if}\ N=2;\\
    &&\sup_{x\in B_1(x_0)}\int_{\abs{y-x}\leq 1}\abs{x-y}^{2-N} {\widehat V(y)}^-\, dy\quad
    \mbox{if}\ N\geq 3.
    \end{eqnarray*}
\end{lem}

\begin{proof}[Proof of (VI) of Lemma \ref{lem:5.1}]
We may assume that $r_0\in (0,1]$ and $\norm u_{H^1(B_2(x_0))}\leq r_0\leq 1$.

We apply Lemma \ref{lem:A.1} to obtain
	$$	\abs{u(x)} \leq C\int_{B_1(x)}\abs{u(y)}\, dy \quad \text{for}\ x\in B_1(x_0),
	$$
where $C>0$ is a constant depending only on 
	$$	\left( V(x)+Q(x)\frac{h(u)}u -Q(x)\frac{f(u)}u\right)^- \leq Q(x)\frac{f(u)}u
		\quad \text{in}\ B_2(x_0).
	$$
Thus $C>0$ does not depend on $u$ with $\norm u_{H^1(B_2(x_0))}\leq 1$.  Therefore we have
	$$	\abs{u(x)} \leq C\norm u_{L^1(B_1(x))} \leq C'r_0.
	$$
Choosing $r_0>0$ small, we get the conclusion.
\end{proof}

%%%%%%%%%%%%%%%%%%%%%%%%%%%%%
%%%%%%%%%%%%%%%%%%%%%%%%%%%%%

\section*{Acknowledgements}
The authors are greatful to Professor Zhi-Qiang Wang for suggesting us to study this problem 
and also for valuable comments.  The Authors are also grateful to to Professor Norihisa Ikoma for 
valuable comments.

This paper was written during Chengxiang Zhang's visit to Department of Mathematics,
School of Science and Engineering, Waseda University as a research fellow with support from
China Scholarship Council.  Chengxiang Zhang would like to thank China Scholarship Council
for the support and Waseda University for kind hospitality.

%%%%%%%%%%%%%%%%%%%%%%%%%%%%%
%%%%%%%%%%%%%%%%%%%%%%%%%%%%%

\end{document}